\documentclass{article}
\usepackage{iclr2027_preprint,times}

\usepackage{amsmath,amssymb,amsthm}
\usepackage{booktabs}
\usepackage{xcolor}
\usepackage{algorithm}
\usepackage{algpseudocode}
\usepackage{enumitem}
\usepackage{hyperref}
\usepackage{url}

\DeclareMathOperator*{\argmin}{arg\,min}
\DeclareMathOperator*{\argmax}{arg\,max}
\DeclareMathOperator{\lmo}{lmo}
\DeclareMathOperator{\KL}{KL}
\DeclareMathOperator{\conv}{conv}
\DeclareMathOperator{\sign}{sign}

\newcommand{\R}{\mathbb{R}}
\newcommand{\E}{\mathbb{E}}
\newcommand{\simplex}{\Delta_K}
\newcommand{\dn}[1]{\|#1\|_*}
\newcommand{\pn}[1]{\|#1\|}
\newcommand{\ip}[2]{\langle #1, #2\rangle}
\newcommand{\gF}{\nabla F}

\newtheorem{theorem}{Theorem}[section]
\newtheorem{lemma}[theorem]{Lemma}
\newtheorem{proposition}[theorem]{Proposition}
\newtheorem{corollary}[theorem]{Corollary}
\newtheorem{assumption}{Assumption}
\theoremstyle{definition}
\newtheorem{definition}{Definition}[section]
\theoremstyle{remark}
\newtheorem{remark}{Remark}[section]
\theoremstyle{plain}

\title{Geometry-Aware Multi-Objective Optimization:
Norm-Constrained Oracles with Entropic Weight Updates}

\author{El Mahdi Chayti, Alireza Mirrokni, Martin Jaggi \\ Machine Learning and Optimization Laboratory (MLO), EPFL}

\iclrfinalcopy  

\begin{document}

\maketitle

\begin{abstract}
Multiple-gradient descent is typically formulated in Euclidean geometry,
whereas matrix-aware optimizers use directions induced by other norms. We
analyze a single-loop multi-objective method combining a norm-constrained
linear minimization oracle with entropic updates of task weights. For any
finite-dimensional norm pair, we establish an $O(T^{-1/2})$ ergodic stationarity
rate in the deterministic setting, deriving the required trajectory gradient
bound from smoothness and lower-boundedness in the unregularized case. With
stochastic gradients and a bounded-gradient assumption, one persistent
blended-gradient momentum buffer yields an $O(T^{-1/4})$ rate whose leading
term has no explicit dependence on the number of tasks. We also establish
static regret for the inner weighting objective on arbitrary bounded gradient
sequences and, with entropy regularization and a weight floor, iteration-wise
inner-gap control after burn-in while retaining ergodic stationarity,
conditional on a trajectory gradient bound. The optimal blended gradient can
be nonunique in spectral and sign geometries, motivating a selection-independent
inner gap. Our analysis extends the spectral construction studied by concurrent
MOON to general norms. Experiments compare the method with existing multi-task
optimizers and examine its oracle, weighting, and momentum choices.
\end{abstract}

\section{Introduction}

Many learning problems optimize several objectives at once: a shared backbone
serving segmentation and depth, a language model balancing helpfulness against
safety. Writing $f_1,\dots,f_K:\R^d\to\R$ for the losses, one seeks a point that
cannot be improved on one objective without degrading another: a Pareto-optimal
point. Pareto stationarity is a first-order necessary condition for such
optimality. A standard algorithm targeting this condition is multiple
gradient descent \citep{desideri2012mgda,fliege2000steepest,sener2018mtl}:
choose the convex combination of task gradients with smallest norm,
\begin{equation}\label{eq:mgda-intro}
w^\star(x) \in \argmin_{w\in\simplex}\ \|\gF(x)w\|_2,
\qquad \gF(x)w := \sum_{i=1}^K w_i \nabla f_i(x),
\end{equation}
and step along $-\gF(x)w^\star(x)$, the \emph{conflict-avoidant} direction
that minimizes $\max_i\ip{\nabla f_i(x)}{d}+\tfrac12\|d\|_2^2$.
A large literature refines
the weighting
\citep{yu2020pcgrad,chen2020graddrop,liu2021cagrad,navon2022nashmtl,liu2023famo,ban2024fairgrad}
and proves rates for stochastic variants
\citep{liu2021smg,zhou2022crmogm,fernando2023moco,chen2024modo,xiao2023sdmgrad,zhang2025mgda}.

Two choices in \eqref{eq:mgda-intro} are Euclidean by inheritance. The \emph{parameter update} is steepest
descent in $\ell_2$, whereas a network's parameters are matrices with their own
scales; optimizers that take an LMO step in an operator norm, as in Muon
\citep{jordan2024muon}, Scion \citep{pethick2025scion} and the duality view of
\citet{bernstein2024anthology,bernstein2024duality,large2024modular}, have
improved large-scale training performance \citep{liu2025moonlight}.
These methods are single-objective; concurrent MOON \citep{moon2026} brings
the spectral construction to MOO. The \emph{weight update} lives on the simplex,
where entropic mirror descent gives an $O(M\sqrt{\log K/T})$ regret bound
under an $\ell_\infty$ subgradient bound $M$, compared with the
$O(M\sqrt{K/T})$ bound obtained by converting that same bound to $\ell_2$
for Euclidean projected descent
\citep{nemirovski1983problem,beck2003mirror,bubeck2015convex}.

The closest methods already use elements of this construction. MSMD
\citep{msmd2024} uses entropic mirror descent on the weights with a Euclidean
direction variable and a double loop. MOON \citep{moon2026} combines a
spectral-norm oracle with logit-based task weighting and establishes the same
deterministic and stochastic rate exponents. Our contribution is a general-norm
analysis, a derived deterministic trajectory gradient bound, and refinements
for blended momentum and inner-problem tracking.

\begin{itemize}[leftmargin=1.4em,itemsep=1pt]
\item \textbf{General-norm deterministic convergence.}
Theorem~\ref{thm:det} gives
$\tfrac1T\sum_t\dn{\gF(x_t)w_t}\le(\sqrt{2\Delta L}+M\sqrt{\log K/2})/\sqrt T$.
For the unregularized, unclipped method, Lemma~\ref{lem:M} derives
$M=O(\sqrt{L\Delta\log K})$ from smoothness and lower-boundedness, yielding
$O(L\Delta\log^2K/\epsilon^2)$ iterations without a separate gradient-bound
assumption. Each iteration uses $K$ task-gradient evaluations, one oracle call,
and one entropic weight update.
\item \textbf{Stochastic convergence with one persistent buffer.}
Under a trajectory gradient bound, Theorem~\ref{thm:stoch} gives
$O(\epsilon^{-4})$ iterations with momentum on the blended gradient.
A weighted-average bound on noisy weight increments keeps explicit $K$
dependence out of the leading $T^{-1/4}$ term; lower-order terms retain linear
$K$ dependence under bounded variance, replaced by $\log K$ under a stronger
tail assumption (Corollary~\ref{cor:subg}).
\item \textbf{Selection-independent inner optimality and iteration-wise tracking.}
The optimal blended gradient can be nonunique in the Muon and Signum
geometries (Proposition~\ref{prop:unique}), motivating the inner gap.
Theorem~\ref{thm:regret} gives static regret on arbitrary bounded gradient
sequences. Average gap and regret are also bounded by the blended-gradient
residual, so alone they do not distinguish adaptation from a converged
scalarization; with regularization, clipping and a trajectory gradient bound,
Theorem~\ref{thm:iterwise} instead controls the inner gap at \emph{every}
iteration after a $\tilde O(M^4\log^2K/\epsilon^4)$ burn-in while retaining
ergodic stationarity, in any norm but with a non-competitive exponent.
\end{itemize}

\section{Related work}

\paragraph{Geometry-aware optimizers.} Non-Euclidean steepest descent is
classical \citep{nesterov2012coordinate,carlson2015spectral}; the current line
begins with \citet{bernstein2024anthology} and Muon \citep{jordan2024muon}.
\citet{pethick2025scion} give the LMO formulation we build on, with
$O(n^{-1/4})$ stochastic rates and layer-wise norm and radius choices;
\citet{bernstein2024duality,large2024modular} develop the duality view;
\citet{liu2025moonlight,riabinin2025gluon} address scaling, and
\citet{yang2021tp5,yang2023spectral} hyperparameter transfer. All are
single-objective.

\paragraph{Stochastic MGDA.} \citet{liu2021smg} give the first stochastic
multi-gradient method; \citet{zhou2022crmogm} show naive stochastic gradient
manipulation need not converge; \citet{fernando2023moco} fix the bias with a
tracking variable; \citet{chen2024modo} show a three-way trade-off between
optimization, generalization and conflict avoidance; \citet{xiao2023sdmgrad}
add a direction-oriented regularizer. \citet{zhang2025mgda} analyze MGDA under
generalized smoothness with no gradient bound, obtaining $O(\epsilon^{-2})$ and
$O(\epsilon^{-4})$ for average CA distance and $O(\epsilon^{-11})$,
$O(\epsilon^{-17})$ for iteration-wise CA distance; these are our reference
point. All are Euclidean with projected weights.

\paragraph{The two closest methods.} \emph{MSMD} \citep{msmd2024} solves the
MGDA saddle problem by stochastic mirror descent with the negative entropy on
the weights, so the $\log K$ benefit of entropic weights is theirs; its
direction variable is Euclidean, it is double-loop, and it assumes bounded
gradients. \emph{MOON} \citep{moon2026} uses spectral--nuclear geometry with
momentum and logit-based task weighting. Its analysis assumes bounded
gradients and losses and carries polynomial dependence on the number of
objectives. Our spectral specialization shares this geometric construction;
our stochastic analysis evaluates the oracle at the previous buffer to
decorrelate the direction from the current minibatch, and our tracking variant
adds a weight floor. Practical implementations also differ in momentum,
logit updates, and auxiliary optimizers, as detailed in
Appendix~\ref{app:expdetails}; the empirical comparison does not isolate a
single change. Our analytical additions are the arbitrary norm pair, the
derived deterministic $M$, comparator-dependent regret, iteration-wise inner
gap, nonuniqueness analysis, and the one-buffer $K$ argument. Conditional-gradient MOO
\citep{assuncao2021conditional} applies an oracle to the constraint set, a
different object. See \citet{gradmoosurvey2025} for a survey and
\citet{nemirovski1983problem,beck2003mirror,kivinen1997eg,bubeck2015convex} for
the entropic mirror-descent bounds we use off the shelf.

\section{Setup}\label{sec:setup}

\paragraph{Norms and the oracle.} Fix a norm $\pn\cdot$ on $\R^d$, the
\emph{primal} norm, measuring parameter displacements, with dual
$\dn g := \sup_{\pn v\le1}\ip g v$ measuring gradients; H\"older's inequality
$\ip g v\le\dn g\pn v$ holds throughout. Let $B:=\{v:\pn v\le R\}$. The
\emph{linear minimization oracle} is any selection
\begin{equation}\label{eq:lmo}
\lmo(g)\in\argmin_{v\in B}\ \ip g v ,
\qquad\text{so that}\qquad
\pn{\lmo(g)}\le R,\quad \ip{g}{\lmo(g)} = -R\dn g .
\end{equation}
The identity holds for every element of the argmin and is the only property of
$\lmo$ our analysis uses; we set $\lmo(0)=0$, so that an exactly zero
blend produces a zero parameter step. With entropy regularization, the weights
can still change; the full state need not be a fixed point. Table~\ref{tab:lmo} lists the instances of interest.
For a network with parameter blocks $x=(x_\ell)_\ell$ the norm used in practice
is $\pn x:=\max_\ell\|x_\ell\|_{(\ell)}/\rho_\ell$, with dual
$\dn g=\sum_\ell\rho_\ell\|g_\ell\|_{(\ell),*}$, blockwise oracle
$\lmo(g)=(\rho_\ell\,\lmo_\ell(g_\ell))_\ell$ of per-block radius $R\rho_\ell$,
Assumption~\ref{as:smooth} imposed in this pair, and
$\zeta^2\le\sum_\ell\rho_\ell^2\zeta_\ell^2$.
The minus sign is inside the oracle: the update is $x_{t+1}=x_t+\gamma D_t$
with $D_t=\lmo(G_t)$, which for the Euclidean ball is normalized gradient
descent, $x-\gamma Rg/\|g\|_2$ \citep{pethick2025scion}.

\begin{table}[t]
\centering\small
\begin{tabular}{llll}
\toprule
\textbf{Primal $\pn\cdot$} & \textbf{Dual $\dn\cdot$} & \textbf{$\lmo(g)$} & \textbf{Known as}\\
\midrule
$\ell_2$ & $\ell_2$ & $-R\,g/\|g\|_2$ & normalized SGD\\
$\ell_\infty$ & $\ell_1$ & $-R\sign(g)$ & Signum \citep{bernstein2018signsgd}\\
spectral & nuclear & $-R\,UV^\top$, $\ G=U\Sigma V^\top$ & Muon \citep{jordan2024muon}\\
\bottomrule
\end{tabular}
\caption{Norm-constrained oracles; \citet{pethick2025scion} apply them
layer-wise with per-layer radii $\rho_\ell$.}
\label{tab:lmo}
\end{table}

\paragraph{MOO in a norm pair.} Write $F(x)=(f_1(x),\dots,f_K(x))^\top$ and,
for $w\in\simplex:=\{w\ge0,\ \sum_iw_i=1\}$, $\gF(x)w:=\sum_iw_i\nabla f_i(x)$
and $F(x)w:=\sum_iw_if_i(x)$.

\begin{definition}[Stationarity, CA distance, inner gap]\label{def:ca}
$x$ is \emph{$\epsilon$-accurate Pareto stationary} if
$\min_{w\in\simplex}\dn{\gF(x)w}\le\epsilon$; $w^\star(x)$ denotes a minimizer
and $d^\star(x):=\gF(x)w^\star(x)$. For an iterate--weight pair $(x_t,w_t)$,
the \emph{CA distance} is $\mathrm{CA}_t:=\dn{\gF(x_t)w_t-d^\star(x_t)}$ and the
\emph{inner gap} is
$\mathrm{gap}_t:=\dn{\gF(x_t)w_t}-\min_{w\in\simplex}\dn{\gF(x_t)w}\ge0$. By the
reverse triangle inequality $\mathrm{gap}_t\le\mathrm{CA}_t$ for every selection
$w^\star(x_t)$; $\mathrm{gap}_t=0$ iff $w_t$ solves the inner problem; and,
unlike $\mathrm{CA}_t$, the gap involves no selection.
\end{definition}

Two facts about these objects differ from the Euclidean intuition
(proofs in Appendix~\ref{app:saddle}).

\begin{proposition}[Uniqueness of the optimal blended gradient]\label{prop:unique}
If the dual unit ball is strictly convex, $d^\star(x)$ is the same for every
minimizer $w^\star(x)$, for every instance. If it is not, there is an instance
with $K=2$ in which every $w\in\simplex$ is optimal and $d^\star$ ranges over a
whole segment: take $\nabla f_1=a$, $\nabla f_2=b$ for distinct unit vectors
$a,b$ with $\dn{(a+b)/2}=1$ (for $\ell_1$: $a=(1,0)$, $b=(0,1)$).
\end{proposition}

\begin{proposition}[Stationary points of fixed scalarizations are Pareto stationary]\label{prop:scalarization}
If $\bar x$ is a first-order stationary point of $x\mapsto F(x)u$ for a fixed
$u\in\simplex$, i.e. $\gF(\bar x)u=0$, then $\min_{w\in\simplex}\dn{\gF(\bar x)w}=0$,
attained at $w=u$.
\end{proposition}

Strict convexity holds for $\ell_p$, $1<p<\infty$, and fails for $\ell_1$ and
the nuclear norm, so in exactly the Muon and Signum geometries the CA distance
can depend on the selected minimizer; our guarantees hold for any measurable
selection or, better, for the selection-free gap. Throughout, $K\ge2$.
To distinguish Pareto stationarity from the residual of the weights actually
used, write $\psi_t(w):=\dn{\gF(x_t)w}$ and
$S_T:=T^{-1}\sum_t\psi_t(w_t)$. Directly from nonnegativity,
\begin{equation}\label{eq:gap-stationarity}
\frac1T\sum_t\mathrm{gap}_t\le S_T,\qquad
\frac1T\Big[\sum_t\psi_t(w_t)-\min_{u\in\simplex}\sum_t\psi_t(u)\Big]\le S_T.
\end{equation}
Also $\mathrm{CA}_t\le2\psi_t(w_t)$ in every norm; in Hilbert dual geometry,
$\mathrm{CA}_t^2\le\psi_t(w_t)^2-\psi_t(w^\star(x_t))^2$
\citep{zhang2025mgda}. Thus convergence of the actual blended-gradient
residual implies average-gap, average-CA, and static-regret upper bounds,
including for a converged fixed scalarization. Smallness of
$\min_w\psi_t(w)$ alone does not control a separately chosen $w_t$.

The gap is useful as a selection-independent inner optimality measure, but
does not certify simultaneous per-task descent for every oracle selection.
Appendix~\ref{app:saddle} gives the saddle condition and a sign-oracle
counterexample. The additional results are a comparator-dependent regret
bound valid independently of outer progress (Theorem~\ref{thm:regret}) and an
every-iterate tracking guarantee after burn-in (Theorem~\ref{thm:iterwise}).

\paragraph{Assumptions.}
\begin{assumption}[Smoothness in the norm pair]\label{as:smooth}
Each $f_i$ is differentiable and $\dn{\nabla f_i(x)-\nabla f_i(y)}\le L\pn{x-y}$.
\end{assumption}
\begin{assumption}[Lower bound]\label{as:lb}
$f_i^\star:=\inf_xf_i(x)>-\infty$; $\Delta:=\max_i\big(f_i(x_0)-f_i^\star\big)>0$
and $L>0$ (if $\Delta=0$ or $L=0$ every gradient vanishes at $x_0$ and there is
nothing to prove; excluding these keeps $M>0$ and the step sizes defined).
\end{assumption}
\begin{assumption}[Stochastic oracle]\label{as:noise}
With $\mathcal F_t$ the history up to the computation of $x_t,w_t$, the sample
$\xi_t$ is independent of $\mathcal F_t$ and
$\E[\nabla f_i(x_t;\xi_t)\mid\mathcal F_t]=\nabla f_i(x_t)$,
$\E[\|\nabla f_i(x_t;\xi_t)-\nabla f_i(x_t)\|_2^2\mid\mathcal F_t]\le\sigma^2$
for every $i$ (a shared minibatch across tasks is allowed).
\end{assumption}
\begin{assumption}[Gradient bound; stochastic setting only]\label{as:gradbound}
$\max_i\dn{\nabla f_i(x_t)}\le M$ a.s.\ along the trajectory of Algorithm~\ref{alg:stoch}.
\end{assumption}
Assumption~\ref{as:noise} is in $\ell_2$ because the variance recursion needs
an inner product; we convert through $\zeta:=\max_{x\ne0}\dn x/\|x\|_2$
($1$ for $\ell_2$, $\sqrt d$ for $\ell_1$, $\sqrt{\min(m,n)}$ for the nuclear
norm on $m\times n$ matrices), as
in \citet{pethick2025scion}. In the deterministic analysis $M$ is derived, not
assumed (Lemma~\ref{lem:M}).

\section{Entropic LMO-MGDA}\label{sec:alg}

\begin{algorithm}[t]
\caption{Entropic LMO-MGDA (with optional regularization $\tau\ge0$ and clipping $\kappa\ge0$)}
\label{alg:main}
\begin{algorithmic}[1]
\Require $x_0\in\R^d$; step sizes $\gamma,\eta>0$; radius $R$; regularization $\tau\ge0$; clipping $\kappa\in[0,1/K]$
\State $w_0\gets(1/K,\dots,1/K)$
\For{$t=0,1,\dots,T-1$}
  \State compute $\nabla f_i(x_t)$ for $i\in[K]$
    \Comment{$K$ backward passes}
  \State $G_t\gets\sum_{i=1}^K w_{t,i}\nabla f_i(x_t)$;\quad $D_t\gets\lmo(G_t)$
    \Comment{one oracle call, layer-wise}
  \State $g_{t,i}\gets\ip{D_t}{\nabla f_i(x_t)}$ for $i\in[K]$
  \State $w_{t+1}\gets\displaystyle\argmin_{w\in\simplex^\kappa}\Big\{-\eta\ip{g_t}{w}+\eta R\tau\,\KL(w\|u_K)+\KL(w\|w_t)\Big\}$
    \Comment{softmax then clip, \eqref{eq:regeg}}
  \State $x_{t+1}\gets x_t+\gamma D_t$
\EndFor
\end{algorithmic}
\end{algorithm}

Line~6 is entropic mirror descent on the scores $g_t$, with an optional pull of
strength $\tau$ toward the uniform weighting $u_K$ and an optional floor
$\kappa$ on every weight, $\simplex^\kappa:=\{w\in\simplex:w_i\ge\kappa\}$. The
sign is correct and counterintuitive: $-g_t/R$ is a subgradient of
$w\mapsto\dn{\gF(x_t)w}$ at $w_t$ (Lemma~\ref{lem:subgrad}), so descending on
it ascends in $g_t$: a task well aligned with $D_t$ has $g_{t,i}\ll0$ and
loses weight, while a task being harmed gains it. The closed form is a softmax
followed by clipping (Lemma~\ref{lem:closedform}): with $\rho:=\eta R\tau$,
\begin{equation}\label{eq:regeg}
w_{t+1,i}=\max\Big\{\kappa,\ \frac{w_{t,i}^{1/(1+\rho)}\,e^{\eta g_{t,i}/(1+\rho)}}{Z_t}\Big\},
\qquad Z_t\text{ such that }\textstyle\sum_iw_{t+1,i}=1,
\end{equation}
where $Z_t$ and the set of clipped coordinates are found by a monotone
recurrence in at most $K$ rounds (Appendix~\ref{app:clip}). At $\tau=\kappa=0$
this is exponentiated gradient, $w_{t+1,i}\propto w_{t,i}e^{\eta g_{t,i}}$;
the entropy regularizer shrinks the logits toward a uniform weighting,
related to MOON's logit decay \citep{moon2026}. The
method costs $K$ backward passes, one oracle call and one softmax per
iteration (for $\kappa>0$, Appendix~\ref{app:clip} gives a clipping recurrence with
$O(K^2)$ worst-case cost, or $O(K\log K)$ using sorting), stores the $K$ task gradients within an
iteration (the scores need $D_t$), and needs no inner optimization loop. Sections~\ref{sec:det}--\ref{sec:stoch}
analyse $\tau=\kappa=0$; Section~\ref{sec:iterwise} shows that $\tau,\kappa>0$
make the weights track the inner optimum at every iterate after a burn-in,
without entering the stationarity inequality.

\subsection{Deterministic analysis}\label{sec:det}

Three lemmas, proved in Appendix~\ref{app:det}: the descent lemma in the norm
pair turned into a one-step inequality by the oracle identity, the regularized
mirror-descent bound, and a self-consistent bound on the gradients along the
trajectory.

\begin{lemma}[One-step inequality]\label{lem:onestep}
Under Assumption~\ref{as:smooth}, for every $w\in\simplex$ and every $t$,
\begin{equation}\label{eq:onestep}
F(x_{t+1})w\;\le\;F(x_t)w\;-\;\gamma R\,\dn{G_t}\;+\;\gamma\ip{g_t}{w-w_t}\;+\;\tfrac12L\gamma^2R^2 .
\end{equation}
\end{lemma}

The descent $-\gamma R\dn{G_t}$ is \emph{linear} in the dual norm, not
quadratic as for gradient descent; this is why our stationarity criterion is
first-moment, as in \citet{pethick2025scion}.

\begin{lemma}[Regularized entropic mirror descent]\label{lem:md}
Line~6 of Algorithm~\ref{alg:main} satisfies, for every $u\in\simplex^\kappa$,
with $\rho=\eta R\tau$ and $r(w):=\KL(w\|u_K)$,
\begin{equation}\label{eq:md}
\eta\ip{g_t}{u-w_t}\;\le\;\KL(u\|w_t)-(1+\rho)\KL(u\|w_{t+1})
+\rho\big(r(u)-r(w_{t+1})\big)+\tfrac12\eta^2\|g_t\|_\infty^2 .
\end{equation}
For $u=u_K$ the regularization terms are $\le0$; for general $u$ they are
$\le\rho\ln K$; at $\tau=\kappa=0$ this is the standard bound for every
$u\in\simplex$.
\end{lemma}

\begin{lemma}[Trajectory gradient bound]\label{lem:M}
Under Assumptions~\ref{as:smooth}--\ref{as:lb}, run Algorithm~\ref{alg:main}
with $\tau=\kappa=0$, $\gamma R=\sqrt{2\Delta/(LT)}$ and
$\eta R=\sqrt{2\log K}/(M\sqrt T)$. If
$M\ge2\sqrt{L\Delta\log K}+2\sqrt{L\Delta(\log K+1)}$ then
$\dn{\nabla f_i(x_t)}\le M$ for all $t\le T$ and $i\in[K]$. In particular
$M=O(\sqrt{L\Delta\log K})$, the larger root of an explicit quadratic, suffices,
with no circularity.
\end{lemma}

\begin{theorem}[Deterministic rate]\label{thm:det}
Under Assumptions~\ref{as:smooth}--\ref{as:lb}, run Algorithm~\ref{alg:main}
with any $\tau\ge0$, $\kappa\in[0,1/K]$, and suppose
$\max_i\dn{\nabla f_i(x_t)}\le M$ for $t<T$. Then for every $\gamma,\eta>0$,
\begin{equation}\label{eq:det-master}
\frac1T\sum_{t=0}^{T-1}\dn{\gF(x_t)w_t}\;\le\;\frac{\Delta}{\gamma RT}+\frac{\eta RM^2}{2}+\frac{L\gamma R}{2},
\end{equation}
in which neither $\tau$, $\kappa$ nor $K$ appears. With the step sizes of
Lemma~\ref{lem:M},
\begin{equation}\label{eq:det}
\frac1T\sum_{t=0}^{T-1}\dn{\gF(x_t)w_t}
\;\le\;\frac{\sqrt{2\Delta L}+M\sqrt{\log K/2}}{\sqrt T},
\end{equation}
so with $\hat t\in\argmin_{t<T}\dn{\gF(x_t)w_t}$ the output $x_{\hat t}$ is
$\epsilon$-accurate Pareto stationary,
$\min_w\dn{\gF(x_{\hat t})w}\le\dn{\gF(x_{\hat t})w_{\hat t}}\le\epsilon$,
after $T=O\big((\Delta L+M^2\log K)/\epsilon^2\big)$ iterations; for
$\tau=\kappa=0$, $M$ is supplied by Lemma~\ref{lem:M}, and substituting
$M^2=O(L\Delta\log K)$ gives $T=O(L\Delta\log^2K/\epsilon^2)$. No property of the
primal norm beyond Assumption~\ref{as:smooth} is used.
\end{theorem}

The proof compares against $u_K$, for which $\KL(u_K\|w_0)=0$ and the
regularization only helps: given the same trajectory gradient bound $M$,
$\tau,\kappa$ do not enter the inequality (they may change the trajectory and
hence the smallest valid $M$), and $\log K$ enters \eqref{eq:det} only through
$\eta$, as the price of adaptivity rather than of stationarity.

Since $\mathrm{CA}_t\le2\dn{\gF(x_t)w_t}$ by the triangle inequality, the
average CA distance obeys twice \eqref{eq:det} for every selection
(Corollary~\ref{cor:ca}); equally for a fixed weighting, so this is a
consistency statement, not evidence of tracking. The same stationarity bound
controls the average inner gap by \eqref{eq:gap-stationarity}. The next result
gives a comparator-dependent guarantee without requiring outer progress.

\begin{theorem}[Inner regret on bounded gradient sequences]\label{thm:regret}
Let $A_t:\R^K\to\R^d$, $t<T$, be any sequence of linear maps with
$\max_i\dn{A_te_i}\le M$, where $M>0$. Start from $w_0=u_K$ and set
$D_t=\lmo(A_tw_t)$, $g_t=A_t^\top D_t$, and
$w_{t+1,i}\propto w_{t,i}e^{\eta g_{t,i}}$ with
$\eta R=\sqrt{2\log K}/(M\sqrt T)$. Then, for every $u\in\simplex$,
\begin{equation}\label{eq:regret}
\frac1T\sum_{t=0}^{T-1}\dn{A_tw_t}
\;\le\;\frac1T\sum_{t=0}^{T-1}\dn{A_tu}
+M\sqrt{\frac{2\log K}{T}}.
\end{equation}
In particular, this holds for $A_t=\gF(x_t)$ under the conditions of
\eqref{eq:det} with $\tau=\kappa=0$.
\end{theorem}

The subgradient inequality gives
$\dn{A_tw_t}-\dn{A_tu}\le\ip{g_t}{u-w_t}/R$, and
Lemma~\ref{lem:md} sums the right-hand side. The proof is pathwise, allows
adaptively chosen sequences, and needs neither smooth losses nor outer descent.
The comparator is evaluated on the algorithm's own trajectory; on a prescribed
sequence with constant columns $a,-a$, a fixed $e_1$ has regret $\dn a$ against
the uniform weighting, which says nothing about scalarization on its own
trajectory.

\begin{corollary}[Static regret to gap decomposition]\label{cor:gap}
Under the conditions of Theorem~\ref{thm:regret} with $A_t=\gF(x_t)$,
\begin{equation}\label{eq:gap}
\frac1T\sum_{t=0}^{T-1}\mathrm{gap}_t\;\le\;M\sqrt{\frac{2\log K}{T}}\;+\;V_T,
\qquad
V_T:=\min_{u\in\simplex}\frac1T\sum_{t=0}^{T-1}
\Big[\dn{\gF(x_t)u}-\min_{w\in\simplex}\dn{\gF(x_t)w}\Big],
\end{equation}
where the comparator-variation term $V_T\ge0$ is the shortfall of the best
\emph{fixed} weighting against the per-iterate optimum, a property of the
trajectory rather than of the update. This decomposition is informative when
$V_T$ is controlled. Along the trajectory covered by Theorem~\ref{thm:det},
\eqref{eq:gap-stationarity} already gives average-gap convergence at the
stationarity rate, without requiring $V_T\to0$; one may take the smaller of
the two bounds. In the Euclidean case $\mathrm{CA}_t^2\le2M\,\mathrm{gap}_t$, and under the uniform
convexity inequality \eqref{eq:S2} of Appendix~\ref{app:lp},
$\mathrm{CA}_t^q\le(q/\mu)M^{q-1}\mathrm{gap}_t$ (Remark~\ref{rem:hilbert}).
\end{corollary}

\subsection{Stochastic analysis with momentum on the blended gradient}
\label{sec:stoch}

The oracle is nonlinear, so single-minibatch noise biases the direction at any
fixed batch size. Averaging before the oracle call controls this error
\citep{pethick2025scion}; our analysis uses momentum with a
horizon-dependent averaging parameter. Stochastic MGDA methods average each task
gradient separately \citep{fernando2023moco,xiao2023sdmgrad}, storing $K$
buffers; the alternative, which any implementation handing the MGDA direction
to a momentum optimizer already uses, blends first:
\begin{equation}\label{eq:blend}
\widehat G_t=\sum_{i=1}^Kw_{t,i}\nabla f_i(x_t;\xi_{t,i}),
\qquad
\mathcal M_t=(1-\alpha)\mathcal M_{t-1}+\alpha\widehat G_t,
\qquad
D_t=\lmo(\mathcal M_{t-1}).
\end{equation}
This is Algorithm~\ref{alg:stoch} (Appendix~\ref{app:stoch}). The tracking
error is one vector $\lambda_t:=\mathcal M_t-\gF(x_t)w_t$ whose noise part is
$K$-free by Jensen, $\E\|\sum_iw_{t,i}n_{t,i}\|_2^2\le\sum_iw_{t,i}\E\|n_{t,i}\|_2^2\le\sigma^2$
(the per-task bound in Appendix~\ref{app:pertask} carries $\sqrt K$). Evaluating the
oracle at $\mathcal M_{t-1}$ makes $D_t$ measurable before $\xi_t$, so the
score-noise cross term has zero conditional mean. The buffer chases a target
moving with $w_t$, which moves by \emph{noisy} scores; the pathwise bound
$\|w_{t+1}-w_t\|_1\le2\eta\|g_t\|_\infty$ would bring in a $\sqrt K$, and the
next lemma avoids it.

\begin{lemma}[Weight increment under noisy scores]\label{lem:wstep}
Let $w_{t+1}$ be obtained from $w_t$ by line~6 of Algorithm~\ref{alg:stoch}.
Pathwise, $\|w_{t+1}-w_t\|_1\le2\eta\|g_t\|_\infty$. If
Assumptions~\ref{as:noise}--\ref{as:gradbound} hold and $\eta RM\le\frac12$,
\[
\E\big[\|w_{t+1}-w_t\|_1\mid\mathcal F_t\big]\;\le\;
\omega:=16\,\eta R\,(M+\zeta\sigma)+8K\eta^2R^2\zeta^2\sigma^2 .
\]
\end{lemma}

$K$ appears only at second order in $\eta$: a task can move the weights by a
large amount only if its noisy score is large, which has probability $O(\eta^2)$.

\begin{lemma}[Tracking error]\label{lem:track}
Under Assumptions~\ref{as:smooth}--\ref{as:gradbound} and $\eta RM\le\frac12$,
splitting $\lambda_t=\nu_t+\delta_t$ into noise and drift,
$\E\|\nu_t\|_2^2\le\sigma^2(\alpha+(1-\alpha)^{2t})$ and
$\E\dn{\delta_t}\le(L\gamma R+M\omega)/\alpha$ (the $x$-part admits a deterministic bound
since $\pn{x_t-x_{t-1}}\le\gamma R$ pathwise), so
$\E\dn{\lambda_t}\le\zeta\sigma(\sqrt\alpha+(1-\alpha)^t)+(L\gamma R+M\omega)/\alpha$.
\end{lemma}

\begin{theorem}[Stochastic rate]\label{thm:stoch}
Under Assumptions~\ref{as:smooth}--\ref{as:gradbound}, run
Algorithm~\ref{alg:stoch} with $T\ge(2/M)^{4/3}$, $\alpha=T^{-1/2}$,
$\gamma R=T^{-3/4}$ and $\eta R=T^{-3/4}/M^2$. Then
\begin{equation}\label{eq:stoch}
\frac1T\sum_{t=0}^{T-1}\E\,\dn{\gF(x_t)w_t}
\;\le\;\frac{\Delta+2L+32+2\zeta\sigma(1+16/M)}{T^{1/4}}
+\frac{C_2}{T^{1/2}}+\frac{C_3}{T^{3/4}}+\frac{C_4}{T},
\end{equation}
with $C_2=2\zeta\sigma$, $C_3=1+\tfrac52L+32(1+\zeta\sigma/M)+\zeta^2K\sigma^2/M^2$
and $C_4=2M+32K\zeta^2\sigma^2/M^3$. Writing $C_1$ for the numerator of the
first term, the right-hand side is at most $\epsilon$ once
$T\ge\max\{(4C_1/\epsilon)^4,(4C_2/\epsilon)^2,(4C_3/\epsilon)^{4/3},4C_4/\epsilon\}$;
the average CA distance obeys the same bound up to a factor $2$ for every
measurable selection, and for $\hat t$ uniform on $\{0,\dots,T-1\}$ drawn after
the run, $\E\min_w\dn{\gF(x_{\hat t})w}\le\E\dn{\gF(x_{\hat t})w_{\hat t}}$ is
bounded by \eqref{eq:stoch}. $C_1$ is free of $K$; $K$ enters $C_3$ through the
$\ell_\infty$ error term of mirror descent on noisy scores and $C_4$ through
Lemma~\ref{lem:wstep}. The $K$-free requirement $T\ge(4C_1/\epsilon)^4$
determines the count whenever $C_3\le16C_1^3/\epsilon^2$ and
$C_4\le64C_1^4/\epsilon^3$, i.e. for $K\lesssim\epsilon^{-2}M^2/(\zeta^2\sigma^2)$
as $\epsilon\to0$.
\end{theorem}

Here ``$K$-free'' means no explicit dependence with $L,M,\Delta,\zeta,\sigma$
fixed (these can change with task scaling and norm choices). By
\eqref{eq:gap-stationarity}, \eqref{eq:stoch} also bounds the expected average
inner gap. Two corollaries in Appendix~\ref{app:stoch} give further
comparator and noise-tail bounds.
Corollary~\ref{cor:regret-stoch} bounds the stochastic inner regret against any
deterministic $u$ by \eqref{eq:stoch} with $\Delta\to M^2\ln K$, hence the
average expected gap up to
$\bar V_T:=\min_u\E\frac1T\sum_t[\dn{\gF(x_t)u}-\min_w\dn{\gF(x_t)w}]\ge\E V_T$.
Corollary~\ref{cor:subg} shows that under a sub-Gaussian tail bound the
linear-$K$ terms of $C_3,C_4$ become $\ln K$, so every constant is
polylogarithmic in $K$.

The condition $T\ge(2/M)^{4/3}$ is $\eta RM\le\frac12$. The gradient bound is
assumed here because Lemma~\ref{lem:M}'s induction controls $f_i$ along a
deterministic path, whereas the stochastic one-step inequality has terms
controlled only in expectation (Remark~\ref{rem:gradbound}). The weight step
is $\Theta(T^{-3/4}/M^2)$ rather than $\Theta(T^{-1/2}/M)$ because the latter
leaves a residual $\Theta(M)$ in the drift; this costs a factor $M\sqrt{\log K}$
in the regret constant only (Appendix~\ref{app:pertask} compares with per-task
tracking).

\section{Iteration-wise inner gap}\label{sec:iterwise}

A per-iterate guarantee, which \citet{zhang2025mgda} obtain for the CA
distance in the Euclidean case at $O(\epsilon^{-11})$ and $O(\epsilon^{-17})$,
is stronger than an ergodic one; to track a \emph{moving} target we make the
inner problem strongly convex. Here $\tau>0$ makes $\tilde\psi_t(w):=\dn{\gF(x_t)w}+\tau\KL(w\|u_K)$
strongly convex with a unique minimizer $p_t$ over $\simplex^\kappa$, and
$\kappa>0$ keeps the divergence to that target finite; neither enters
\eqref{eq:det-master}. The analysis uses only three norm-free facts: $-g_t/R$
is a subgradient of $w\mapsto\dn{\gF(x_t)w}$, that function is $M$-Lipschitz in
$\|\cdot\|_1$, and it moves by at most $L\pn{x-x'}$ when $x$ moves.

\begin{theorem}[Iteration-wise inner gap, with progress]\label{thm:iterwise}
Let Assumptions~\ref{as:smooth}--\ref{as:lb} hold and suppose
$\max_i\dn{\nabla f_i(x_t)}\le M$ along the trajectory of
Algorithm~\ref{alg:main}. Fix a tracking accuracy $\epsilon\in(0,1]$ with
$\epsilon\le2M$ and a bias accuracy $\beta\in(0,\epsilon]$, and set
\[
\tau=\frac{\beta}{2\ln K},\qquad
\kappa=\frac{\beta}{4MK},\qquad
\eta R=\frac{\epsilon^{2}\beta}{768\,M^{4}\ln K},\qquad
\gamma R=\frac{\epsilon^{6}\beta^{6}}{27\cdot2^{28}\,(2+\Lambda)\,L\,M^{11}K\ln^{5}K},
\]
with $\Lambda:=\ln(4MK/\beta)$. Then, for any primal norm:
\begin{enumerate}[leftmargin=1.8em,itemsep=1pt]
\item[(i)] \textbf{Tracking.} For every
$t\ge T_0:=\big\lceil\tfrac{6144\,M^4\ln^2K}{\epsilon^2\beta^2}\ln\tfrac{24M^2\ln K}{\epsilon^2}\big\rceil$:
$\ \|w_t-p_t\|_1\le\epsilon/(2M)$ and
$\mathrm{gap}(p_t):=\dn{\gF(x_t)p_t}-\min_w\dn{\gF(x_t)w}\le\beta$, hence
$\mathrm{gap}_t\le\epsilon/2+\beta$. With $\beta=\epsilon/2$:
$\mathrm{gap}_t\le\epsilon$ for all $t\ge T_0=\tilde O(M^4\ln^2K/\epsilon^4)$.
\item[(ii)] \textbf{Progress.} For every $T$, \eqref{eq:det-master} holds, and
its right-hand side is at most $\epsilon/2$ for
$T\ge T_1:=\lceil4\Delta/(\gamma R\,\epsilon)\rceil$; with $\beta=\epsilon/2$,
$T_1=\tilde O(\Delta LM^{11}K/\epsilon^{13})$.
\item[(iii)] \textbf{Tail output.} For $T\ge\max\{2T_0,T_1\}$ and
$\hat t\in\argmin_{T_0\le t<T}\dn{\gF(x_t)w_t}$, the single iterate
$(x_{\hat t},w_{\hat t})$ satisfies both $\dn{\gF(x_{\hat t})w_{\hat t}}\le\epsilon$,
hence $\min_w\dn{\gF(x_{\hat t})w}\le\epsilon$, and $\mathrm{gap}_{\hat t}\le\epsilon/2+\beta$.
\end{enumerate}
If $\epsilon>2M$ all conclusions are trivial.
\end{theorem}

\begin{corollary}[Iteration-wise CA distance under uniform convexity]\label{cor:iterwise-ca}
If the dual norm satisfies \eqref{eq:S2} with power type $q$ and constant $\mu$
(all $\ell_p$, Schatten-$p$, $1<p<\infty$; Proposition~\ref{prop:lp}), then
$\mathrm{CA}_t\le M\|w_t-p_t\|_1+\big((q/\mu)M^{q-1}\mathrm{gap}(p_t)\big)^{1/q}$
for every selection, so Theorem~\ref{thm:iterwise} with
$\beta=\mu\epsilon^q/(q2^qM^{q-1})$ gives $\mathrm{CA}_t\le\epsilon$ for $t\ge T_0$;
in the Euclidean case $T_0=\tilde O(M^6\ln^2K/\epsilon^6)$, $T_1=\tilde O(\Delta LM^{17}K/\epsilon^{19})$.
\end{corollary}

The $\epsilon^{-13}$ bound concerns inner gap, whereas the
$\epsilon^{-11}$ deterministic bound of \citet{zhang2025mgda} concerns
Euclidean CA distance under different smoothness assumptions. For the latter
metric our conversion gives $\tilde O(\epsilon^{-19})$, so we claim no
competitive exponent. The content is an
iteration-wise inner-gap guarantee with simultaneous ergodic progress and one
tail-selected iterate carrying both, in any norm, conditional on the gradient
bound $M$ (a hypothesis here, unlike in Theorem~\ref{thm:det};
Remark~\ref{rem:M-reg}). The tracked object is the regularized floored target
$p_t$; the gap to the original inner problem comes through $\beta$.

Why we track the unsquared objective, unlike the Euclidean analyses, and what
it costs, is explained in Remark~\ref{rem:unsquared}; proofs are in
Appendix~\ref{app:reg}.

\section{Experiments}\label{sec:exp}

We evaluate on Cityscapes ($K=2$: segmentation and depth) and NYU-v2 ($K=3$:
segmentation, depth, surface normals) with the SegNet/MTAN protocol of
\citet{navon2022nashmtl,liu2023famo}, against single-task learning (STL),
linear scalarization (LS), SI, RLW, DWA and UW weighting, MGDA (cold and
warm-started), PCGrad, GradDrop, CAGrad, MoCo, MoDo, Nash-MTL, FAMO and MOON. Our
method is Algorithm~\ref{alg:stoch} with the spectral oracle on matrix-valued
parameters. Tables~\ref{tab:city} and~\ref{tab:nyu} report per-task metrics and
$\Delta m\%$, the mean relative degradation against the single-task (STL) row
of each table (lower is better). The reported architectures, hyperparameters, and
tuning settings are in Appendix~\ref{app:expdetails}.
The dense-prediction results use one seed, averaging the final ten epochs;
they do not quantify variation across independent training runs. Our step sizes
and schedule were tuned on the split on which results are reported, whereas the
baselines use their reference hyperparameters (Appendix~\ref{app:densedetails}).

\begin{table}[t]
\centering
\scriptsize
\setlength{\tabcolsep}{3.5pt}
\begin{tabular}{lccccc}
\toprule
Method & mIoU$\uparrow$ & Pix Acc$\uparrow$ & Abs Err$\downarrow$ & Rel Err$\downarrow$ & $\Delta m\%\downarrow$\\
\midrule
STL & 74.11 & 93.44 & 0.0137 & 38.22 & 0.00\\
\midrule
LS & 74.80 & 93.37 & 0.0188 & 42.41 & 11.83\\
SI & 67.77 & 89.97 & 0.0151 & 34.65 & 3.29\\
RLW & 74.24 & 93.30 & 0.0169 & 47.09 & 11.63\\
DWA & 74.42 & 93.41 & 0.0159 & 41.94 & 6.35\\
UW & 71.56 & 92.15 & 0.0191 & 38.75 & 11.41\\
MGDA & 73.12 & 92.80 & 0.0157 & 36.03 & 2.72\\
PCGrad & 75.89 & 93.72 & 0.0150 & 45.83 & 6.67\\
GradDrop & 75.85 & 93.70 & 0.0151 & 43.60 & 5.42\\
CAGrad & 75.83 & 93.50 & 0.0145 & 35.89 & -0.66\\
MoCo & 74.37 & 93.17 & 0.0221 & 81.06 & 43.33\\
MoDo & 74.71 & 93.51 & 0.0152 & 44.31 & 6.50\\
Nash-MTL & 74.76 & 93.29 & 0.0144 & 33.47 & -2.01\\
FAMO & 74.19 & 93.28 & 0.0154 & \textbf{30.62} & -1.85\\
MGDA (warm) & 75.64 & 93.60 & 0.0147 & 47.00 & 7.01\\
MOON &
\underline{78.53} &
\underline{94.15} &
\underline{0.0141} &
\underline{31.39} &
\textbf{-5.42}\\
\midrule
\textbf{Entropic LMO-MGDA (ours)} &
\textbf{79.03} &
\textbf{94.48} &
\textbf{0.0136} &
36.29 &
\underline{-3.38}\\
\bottomrule
\end{tabular}
\caption{Results on CityScapes dataset. The best and second-best values are highlighted in \textbf{bold} and \underline{underlined}, respectively. The task-specific metrics and the performance drop $\Delta m\%$ are reported.}
\label{tab:city}
\end{table}

\begin{table}[t]
\centering
\scriptsize
\setlength{\tabcolsep}{3.5pt}
\begin{tabular}{lcccccccccc}
\toprule
& \multicolumn{2}{c}{Segmentation} & \multicolumn{2}{c}{Depth} & \multicolumn{5}{c}{Surface normal} & \\
\cmidrule(lr){2-3}
\cmidrule(lr){4-5}
\cmidrule(lr){6-10}
Method & mIoU$\uparrow$ & Pix Acc$\uparrow$ & Abs$\downarrow$ & Rel$\downarrow$ & Mean$\downarrow$ & Med.$\downarrow$ & $<\!11.25\uparrow$ & $<\!22.5\uparrow$ & $<\!30\uparrow$ & $\Delta m\%\downarrow$\\
\midrule
STL & 38.77 & 63.98 & 0.6350 & 0.2654 & 25.10 & 19.06 & 30.12 & 57.66 & 69.53 & 0.00\\
\midrule
LS & 38.13 & 64.83 & 0.5523 & 0.2334 & 27.79 & 23.56 & 23.03 & 48.23 & 61.64 & 6.76\\
SI & 38.59 & 64.57 & 0.5439 & 0.2339 & 27.57 & 23.43 & 22.81 & 48.45 & 62.18 & 6.32\\
RLW & 38.45 & 64.61 & 0.5814 & 0.2454 & 27.83 & 23.46 & 23.20 & 48.51 & 61.81 & 7.53\\
DWA & 39.57 & 65.19 & 0.5691 & 0.2399 & 27.40 & 22.87 & 24.26 & 49.61 & 62.81 & 5.37\\
UW & 38.48 & 63.88 & 0.5555 & 0.2294 & 27.47 & 23.38 & 23.18 & 48.60 & 62.29 & 6.23\\
MGDA & 33.79 & 62.12 & 0.5906 & 0.2398 & \underline{24.47} & \textbf{18.65} & \textbf{31.28} & \textbf{58.50} & \textbf{70.28} & -1.33\\
PCGrad & 38.79 & 65.08 & 0.5501 & 0.2212 & 27.18 & 22.74 & 24.10 & 49.93 & 63.36 & 4.23\\
GradDrop & 40.37 & 66.04 & 0.5536 & 0.2353 & 27.97 & 23.76 & 22.78 & 47.97 & 61.39 & 6.39\\
CAGrad & \underline{41.61} & 66.95 & 0.5444 & 0.2251 & 25.75 & 20.79 & 27.44 & 53.99 & 66.85 & -1.18\\
MoCo & 35.47 & 61.76 & 0.5808 & 0.2328 & 31.21 & 27.97 & 15.31 & 39.38 & 54.01 & 18.38\\
MoDo & 39.41 & 65.82 & 0.5353 & 0.2225 & 25.54 & 20.08 & 28.57 & 55.34 & 67.78 & -1.96\\
Nash-MTL & 38.18 & 65.39 & 0.5575 & 0.2325 & 25.12 & 19.85 & 29.13 & 55.97 & 68.48 & -1.48\\
FAMO & 32.57 & 61.67 & 0.5411 & 0.2220 & 25.16 & 19.25 & 30.18 & 57.16 & 69.10 & -1.00\\
MGDA (warm) & 38.41 & 63.76 & 0.5813 & 0.2453 & 28.13 & 24.03 & 21.43 & 47.24 & 61.13 & 9.15\\
MOON &
39.88 &
\underline{67.30} &
\textbf{0.4951} &
\underline{0.2152} &
25.50 &
20.02 &
29.44 &
55.34 &
67.71 &
\underline{-3.72}\\
\midrule
\textbf{Entropic LMO-MGDA (ours)} &
\textbf{42.71} &
\textbf{68.07} &
\underline{0.5075} &
\textbf{0.2129} &
\textbf{23.76} &
\underline{19.18} &
\underline{30.27} &
\underline{57.34} &
\underline{69.60} &
\textbf{-6.80}\\
\bottomrule
\end{tabular}
\caption{Results on NYU-v2 dataset. The best and second-best values are highlighted in \textbf{bold} and \underline{underlined}, respectively. The task-specific metrics and the performance drop $\Delta m\%$ are reported.}
\label{tab:nyu}
\end{table}

On Cityscapes, our method achieves the highest segmentation scores and the
lowest absolute depth error among the reported methods. MOON achieves a lower
aggregate degradation, $-5.42\%$ versus $-3.38\%$, driven by its lower relative
depth error. On NYU-v2, our method achieves the lowest reported aggregate
degradation, $-6.80\%$ versus $-3.72\%$ for MOON. These are descriptive
comparisons from single runs. With $K=2,3$, they do not establish the predicted
scaling advantage with the number of objectives, and the implementation
differences do not isolate the contribution of geometry or weighting.

\subsection{MultiMNIST: comparison with MOON and ablations}\label{sec:mnist}

We follow MOON's MultiMNIST setup ($K=2$: left and right digit, small ViT).
Because its dataset instance is not released, we regenerate the data and rerun
each method (Appendix~\ref{app:mnistdetails}); only within-table comparisons
are meaningful. The reported mean accuracy is $90.54\%$ for our configuration
and $87.47\%$ for MOON. Hyperparameters for our method, including the weight
step, are selected on a fixed validation split drawn from MNIST-train and
frozen before evaluation on MNIST-test. Every value of the $\eta$ and $\kappa$
grids (Tables~\ref{tab:eta}, \ref{tab:kappa_ablation}) gives at least $89.5\%$,
above all baselines. The spectral oracle is best under the fixed step sizes
used in the ablation, removing momentum reduces accuracy by $8.86$ points,
and entropic and projected weights give similar means at $K=2$; these are
fixed-hyperparameter substitutions, not separately tuned comparisons.
\begin{table}[t]
\centering
\scriptsize

\begin{minipage}[t]{0.48\textwidth}
\centering
\renewcommand{\arraystretch}{0.90}
\setlength{\tabcolsep}{3pt}
\begin{tabular}{lccc}
\toprule
Method & Left$\uparrow$ & Right$\uparrow$ & Avg.$\uparrow$ \\
\midrule
MGDA        & 83.63 & 85.13 & 84.38 \\
MGDA + Muon & 84.48 & 86.43 & 85.45 \\
FAMO        & 85.24 & 87.29 & 86.27 \\
FAMO + Muon & 87.28 & 88.02 & 87.65 \\
Muon (LS)   & \underline{87.45} & 88.04 & \underline{87.75} \\
MOON        & 86.87 & \underline{88.06} & 87.47 \\
\midrule
\textbf{Entropic LMO-MGDA (Ours)}
& \textbf{89.86}
& \textbf{91.23}
& \textbf{90.54} \\
\bottomrule
\end{tabular}
\end{minipage}
\hfill
\begin{minipage}[t]{0.48\textwidth}
\centering
\renewcommand{\arraystretch}{0.85}
\setlength{\tabcolsep}{2.5pt}
\begin{tabular}{llcc}
\toprule
Component & Variant & Avg.$\uparrow$ & Gap$\downarrow$\\
\midrule
Oracle 
& $\ell_2$ & 49.59 & 0.55\\
& sign & 89.01 & 6.38\\
& spectral & 90.54 & 0.90\\
\midrule
Weights 
& entropic & 90.54 & 0.90\\
& projected & 90.51 & 0.97\\
\midrule
Momentum 
& blended & 90.54 & 0.90\\
& per-task & 90.16 & 1.41\\
& none & 81.68 & 1.65\\
\midrule
Step $\eta$ & $10^{-5}$ to $10^{-3}$ & 89.5--90.5 & 0.29--1.97\\
\bottomrule
\end{tabular}
\end{minipage}

\caption{MultiMNIST with a ViT backbone. Left: test accuracy (\%), mean over
three seeds on the same regenerated data. Best and second-best values are shown in bold and underlined.
Right: fixed-hyperparameter substitutions with the final encoder-only inner gap,
in each oracle's dual norm (not comparable across oracle rows). Details in
Appendix~\ref{app:mnistdetails}.}
\label{tab:mnist}
\end{table}

\section{Conclusion}\label{sec:limits}

We analyzed a single-loop entropic-LMO method for multi-objective optimization
in arbitrary norm pairs: $O(\epsilon^{-2})$ deterministic stationarity with a
derived gradient bound, $O(\epsilon^{-4})$ stochastic stationarity with one
persistent buffer and no explicit $K$ in the leading term, static regret of the
weights on bounded gradient sequences, and, with regularization and clipping,
iteration-wise inner-gap control after burn-in.

The inner gap avoids the minimizer-selection dependence of CA distance in
spectral and sign geometries. It does not by itself certify simultaneous
descent, and its average, like average static regret, is already controlled
by the actual blended-gradient residual. The distinct tracking conclusion is
the every-iterate guarantee after burn-in, not average-gap convergence alone.

Several limitations remain. Appendix~\ref{app:inexact} tolerates fixed relative
oracle error for deterministic stationarity and average gap, but not yet for
the stochastic or iteration-wise guarantees. The gradient bound is derived only
in the deterministic unregularized case, and the tracking exponents are not
competitive with Euclidean CA-distance results. Experiments have $K\le3$ and
$\tau=\kappa=0$, so they do not test many-objective scaling or regularized
tracking.

\subsection*{AI use statement}

We used a large language model as an assistant for restructuring and editing the manuscript, checking
and tightening proofs, drafting parts of the experiment code, and suggesting related work. The authors
verified all such output. The research ideas, the method, and the main results are the authors' own,
and the authors take full responsibility for the content of this work.

\subsection*{Reproducibility statement}

All theoretical results are stated together with their assumptions
(Assumptions~\ref{as:smooth}--\ref{as:gradbound}), and complete proofs are in
Appendices~\ref{app:prelim}--\ref{app:inexact}: the structural results in
Appendix~\ref{app:saddle}, the deterministic analysis in
Appendix~\ref{app:det}, the stochastic analysis in Appendix~\ref{app:stoch},
the iteration-wise analysis in Appendix~\ref{app:reg}, the uniform-convexity
conversions in Appendix~\ref{app:lp} and the inexact-oracle analysis in
Appendix~\ref{app:inexact}; explicit constants are given throughout. The
method is fully specified by Algorithms~\ref{alg:main} and~\ref{alg:stoch},
the closed-form weight update~\eqref{eq:regeg} and the clipping recurrence of
Algorithm~\ref{alg:clip}. Appendix~\ref{app:expdetails} describes the
architectures, data construction, reported hyperparameters and tuning grids,
seed counts, validation splits, reporting protocol, ablations, and empirical
encoder-only gap. The $\Delta m\%$ values are computed from the per-task metrics with the reference
values given there. An anonymous implementation of all methods and experiments
is available at \url{https://anonymous.4open.science/r/LMO-MOO-D648/}.

\bibliography{refs}
\bibliographystyle{iclr2027_conference}
\newpage
\appendix
\section{Preliminaries for the proofs}\label{app:prelim}

Throughout, $\pn\cdot$ is the primal norm on $\R^d$, $\dn\cdot$ its dual,
$B=\{v:\pn v\le R\}$, and $\lmo$ is as in \eqref{eq:lmo}. We write
$A:=\gF(x)$ for the linear map $\R^K\to\R^d$, $Aw=\sum_i w_i\nabla f_i(x)$, and
$A^\top$ for its adjoint, so $(A^\top d)_i=\ip{\nabla f_i(x)}{d}$. We record
three elementary facts.

\begin{lemma}[Descent lemma in a norm pair]\label{lem:descent}
Under Assumption~\ref{as:smooth}, for all $x,y\in\R^d$ and each $i$,
\[
f_i(y)\;\le\;f_i(x)+\ip{\nabla f_i(x)}{y-x}+\tfrac{L}{2}\pn{y-x}^2 .
\]
\end{lemma}

\begin{proof}
By the fundamental theorem of calculus along the segment,
\[
f_i(y)-f_i(x)-\ip{\nabla f_i(x)}{y-x}
=\int_0^1\ip{\nabla f_i(x+t(y-x))-\nabla f_i(x)}{y-x}\,dt .
\]
Applying H\"older's inequality $\ip g v\le\dn g\pn v$ and then
Assumption~\ref{as:smooth} with $\pn{(x+t(y-x))-x}=t\pn{y-x}$,
\[
\le\int_0^1\dn{\nabla f_i(x+t(y-x))-\nabla f_i(x)}\pn{y-x}\,dt
\le\int_0^1 Lt\pn{y-x}^2dt=\tfrac{L}{2}\pn{y-x}^2 . \qedhere
\]
\end{proof}

\begin{lemma}[Gradient bound from the function value]\label{lem:gradbound}
Under Assumptions~\ref{as:smooth}--\ref{as:lb},
$\dn{\nabla f_i(x)}^2\le 2L\big(f_i(x)-f_i^\star\big)$ for all $x$ and $i$.
\end{lemma}

\begin{proof}
If $\nabla f_i(x)=0$ there is nothing to prove, so assume $\nabla f_i(x)\ne0$.
Since $\R^d$ is finite-dimensional the primal unit ball is compact, so the
supremum defining $\dn{\nabla f_i(x)}$ is attained: there is $v$ with $\pn v=1$
and $\ip{\nabla f_i(x)}{v}=-\dn{\nabla f_i(x)}$ (negate a maximizer). For $s>0$,
Lemma~\ref{lem:descent} applied at $y=x+sv$ gives
\[
f_i^\star\;\le\;f_i(x+sv)\;\le\;f_i(x)-s\,\dn{\nabla f_i(x)}+\tfrac{L}{2}s^2 .
\]
Minimizing the right-hand side over $s>0$ at $s=\dn{\nabla f_i(x)}/L$ yields
$f_i^\star\le f_i(x)-\dn{\nabla f_i(x)}^2/(2L)$, which rearranges to the claim.
\end{proof}

\begin{lemma}[Subgradient of the weight objective]\label{lem:subgrad}
Fix $x$ and let $\psi(w):=\dn{Aw}$ on $\R^K$. For every $d\in\lmo(Aw)$ we have
$-A^\top d/R\in\partial\psi(w)$; componentwise, $-g/R\in\partial\psi(w)$ where
$g_i=\ip{d}{\nabla f_i(x)}$.
\end{lemma}

\begin{proof}
Because $B=-B$ is symmetric, the dual norm is the support function
$\dn{h}=\tfrac1R\max_{v\in B}\ip h v$, which is convex with
$\partial\dn\cdot(h)=\tfrac1R\argmax_{v\in B}\ip h v
=-\tfrac1R\argmin_{v\in B}\ip h v=-\tfrac1R\lmo(h)$.
The chain rule for the composition of a convex function with a linear map gives
$A^\top\partial\dn\cdot(Aw)\subseteq\partial\psi(w)$, whence the claim.
\end{proof}

Lemma~\ref{lem:subgrad} justifies the sign in line~6 of Algorithm~\ref{alg:main}: since $-g_t/R$ is a
subgradient of $\psi$ at $w_t$, a mirror-descent step \emph{minimizing} $\psi$
moves along $-\eta(-g_t/R)$, i.e. along $+\eta g_t$ after absorbing $R$ into
$\eta$. Note that the lemma holds for \emph{every} element of $\lmo(Aw)$, so no
selection rule is needed for our analysis.

\section{Structural properties of the CA direction}\label{app:saddle}

\begin{proof}[Proof of Proposition~\ref{prop:unique}]
The set $S=\{Aw:w\in\simplex\}=\conv\{\nabla f_1(x),\dots,\nabla f_K(x)\}$ is
convex and compact. Let $m:=\min_{h\in S}\dn h$ and suppose $g_1,g_2\in S$ both
attain $m$ with $g_1\ne g_2$. If $m=0$ then $g_1=g_2=0$, a contradiction; so
$m>0$. Then $g_1/m$ and $g_2/m$ are distinct points of the dual unit sphere, and
strict convexity of the dual unit ball gives
$\dn{\tfrac{g_1+g_2}{2m}}<1$, i.e. $\dn{\tfrac{g_1+g_2}{2}}<m$. But
$\tfrac{g_1+g_2}{2}\in S$ by convexity, contradicting minimality of $m$. Hence
the minimizing value $Aw^\star$ is unique.

For the converse, suppose the dual unit ball is not strictly convex: there
are distinct $a,b$ with $\dn a=\dn b=1$ and $\dn{m}=1$ for $m:=(a+b)/2$. Let
$\varphi$ be a norming functional at $m$, i.e. $\varphi(m)=1$ and
$|\varphi(y)|\le\dn y$ for all $y$ (it exists by compactness of the primal
unit ball, since $\dn m=\max_{\pn v\le1}\ip mv$). Then $\varphi(a)\le1$,
$\varphi(b)\le1$ and $\tfrac12(\varphi(a)+\varphi(b))=\varphi(m)=1$ force
$\varphi(a)=\varphi(b)=1$, so for every $\theta\in[0,1]$,
\[
1=\varphi\big(\theta a+(1-\theta)b\big)\le\dn{\theta a+(1-\theta)b}\le\theta\dn a+(1-\theta)\dn b=1 ,
\]
and the whole segment $[a,b]$ lies on the unit sphere. Take $K=2$ with
$\nabla f_1(x)=a$, $\nabla f_2(x)=b$: then $\dn{Aw}=1$ for every $w\in\simplex$,
every $w$ is optimal, and $Aw^\star$ ranges over $[a,b]$. For $\dn\cdot=\ell_1$
one may take $a=(1,0)$, $b=(0,1)$; for the nuclear norm, $a=E_{11}$,
$b=E_{22}$. Hence strict convexity of the dual ball is necessary and sufficient
for the optimal blended gradient to be unique in every instance.
\end{proof}

\begin{proof}[Proof of Proposition~\ref{prop:scalarization}]
The function $x\mapsto F(x)u=\sum_i u_if_i(x)$ is differentiable with gradient
$\gF(x)u$, so first-order stationarity is $\gF(\bar x)u=0$ (a minimizer is a
special case), and then $\min_{w\in\simplex}\dn{\gF(\bar x)w}\le\dn{\gF(\bar x)u}=0$,
attained at $w=u$.
\end{proof}

We now record the selection criterion mentioned in Section~\ref{sec:setup}. Let
$\Phi_x(w,d)=\ip{\gF(x)w}{d}$ on $\simplex\times B$. Since
$\min_{d\in B}\Phi_x(w,d)=-R\dn{\gF(x)w}$ by definition of the dual norm,
Sion's minimax theorem \citep{sion1958minimax} gives
\begin{equation}\label{eq:saddle}
\max_{w\in\simplex}\min_{d\in B}\Phi_x(w,d)=\min_{d\in B}\max_{w\in\simplex}\Phi_x(w,d)
=-R\min_{w\in\simplex}\dn{\gF(x)w},
\end{equation}
so the CA problem is $\max_w\min_d\Phi_x$, with $w$ the maximizing and $d$ the
minimizing player.

\begin{proposition}[When the oracle output is conflict-avoidant]\label{prop:selection}
Fix $x$, let $w\in\simplex$, $G=\gF(x)w$, and let $D\in\lmo(G)$ with
$g_i=\ip{\nabla f_i(x)}{D}$. Then $(w,D)$ is a saddle point of $\Phi_x$ on
$\simplex\times B$ if and only if
\begin{equation}\label{eq:star}
\max_{i\in[K]}g_i \;=\; -R\,\dn{G} .
\end{equation}
When \eqref{eq:star} holds, every objective improves by at least $R\dn{G}$
along $D$, which by \eqref{eq:saddle} is the optimal guaranteed worst-case
improvement over $B$. If $\dn\cdot$ is differentiable at $G$ then $\lmo(G)$ is a
singleton and \eqref{eq:star} holds automatically for any minimizer $w$ of
$\psi$; if not, \eqref{eq:star} can fail for some elements of $\lmo(G)$.
\end{proposition}

\begin{proof}
$D\in\lmo(G)$ is exactly the statement that $D$ minimizes $\Phi_x(w,\cdot)$ over
$B$. The remaining saddle condition is that $w$ maximizes
$\Phi_x(\cdot,D)=\ip{g}{\cdot}$ over $\simplex$, i.e. that $w$ is supported on
$\argmax_i g_i$, i.e. that $\ip g w=\max_i g_i$. Since
$\ip g w=\ip{\gF(x)w}{D}=\ip{G}{D}=-R\dn G$ by \eqref{eq:lmo}, this is
\eqref{eq:star}. Under \eqref{eq:star} we get $-g_i\ge R\dn G$ for every $i$,
which is the stated improvement. If $\dn\cdot$ is differentiable at $G$ then
$\partial\dn\cdot(G)$ is a singleton, so by Lemma~\ref{lem:subgrad}
$\partial\psi(w)=\{-g/R\}$ and first-order optimality of $w$ on $\simplex$
coincides with \eqref{eq:star}.

For the failure, take the instance in the proof of
Proposition~\ref{prop:unique} with $w=e_1$, so $G=(1,0)$ and
$\lmo(G)=\{(-R,d_2):|d_2|\le R\}$, every element of which satisfies
$\ip G v=-R=-R\|G\|_1$. Here $g=(-R,d_2)$ and \eqref{eq:star} reads
$\max(-R,d_2)=-R$, forcing $d_2=-R$. The selection $d_2=+R$ gives $g_2=R>0$:
the second objective is actively degraded, while the choice $D=(-R,-R)$
improves both by $R$.
\end{proof}

The usual sign selection with $\sign(0)=0$ can also harm a task at an inner
minimizer. Take $R=1$, primal $\ell_\infty$, gradients
$a=(2,-1)$, $b=(-2,3)$, and $w=(1/2,1/2)$. Then $Aw=(0,1)$ and
$\|A(q,1-q)\|_1=|4q-2|+|3-4q|\ge1$ for all $q\in[0,1]$, with equality at
$q=1/2$. Thus the inner gap is zero, but $D=-\sign(Aw)=(0,-1)$ gives
$\ip{a}{D}=1>0$. These gradients occur at the origin for the lower-bounded smooth
quadratics $f_1(x)=\tfrac12\|x\|_2^2+\ip{a}{x}$ and
$f_2(x)=\tfrac12\|x\|_2^2+\ip{b}{x}$. Inner optimality therefore does not alone
certify per-task descent, even for a conventional oracle selection.

We stress that Theorem~\ref{thm:det} and Theorem~\ref{thm:stoch} do not require
\eqref{eq:star}. The reason is that the mirror-descent bound
\eqref{eq:md} holds for an \emph{arbitrary} vector $g_t$, and by
Lemma~\ref{lem:subgrad} every element of $\lmo(G_t)$ produces a valid
subgradient; the only property of $D_t$ used elsewhere is the identity
$\ip{G_t}{D_t}=-R\dn{G_t}$, which likewise holds for every element. Selection
therefore affects the per-iteration guarantee of
Proposition~\ref{prop:selection}, not the average-case rates.

\section{Proofs for the deterministic setting}\label{app:det}

\begin{proof}[Proof of Lemma~\ref{lem:onestep}]
Fix $w\in\simplex$. Multiplying Lemma~\ref{lem:descent} (with $y=x_{t+1}$,
$x=x_t$) by $w_i\ge0$ and summing over $i$, using $\sum_iw_i=1$ for the
quadratic term,
\[
F(x_{t+1})w\;\le\;F(x_t)w+\ip{\gF(x_t)w}{x_{t+1}-x_t}+\tfrac{L}{2}\pn{x_{t+1}-x_t}^2 .
\]
By line~7 of Algorithm~\ref{alg:main}, $x_{t+1}-x_t=\gamma D_t$ with
$\pn{D_t}\le R$, so the last term is at most $\tfrac12L\gamma^2R^2$. For the
inner product, decompose $\gF(x_t)w=G_t+\gF(x_t)(w-w_t)$ and use
\eqref{eq:lmo}:
\[
\ip{\gF(x_t)w}{D_t}=\ip{G_t}{D_t}+\sum_i(w_i-w_{t,i})\ip{\nabla f_i(x_t)}{D_t}
=-R\dn{G_t}+\ip{g_t}{w-w_t}. \qedhere
\]
\end{proof}

\begin{proof}[Proof of Lemma~\ref{lem:md}]
Let $\Phi(w)=\sum_iw_i\ln w_i$ with $\nabla\Phi(w)=(1+\ln w_i)_i$, whose
Bregman divergence is $D_\Phi(u,v)=\KL(u\|v)$ for $u\in\simplex$ and $v$ in the
relative interior (with $0\ln0=0$). Write $C:=\simplex^\kappa$ and
\[
h(w):=-\eta\ip{g_t}{w}+\rho\,r(w)+\KL(w\|w_t),\qquad \rho=\eta R\tau,
\]
so that $w_{t+1}=\argmin_Ch$. Both $r=\KL(\cdot\|u_K)$ and $\KL(\cdot\|w_t)$
have Bregman divergence $\KL(u\|v)$ with respect to themselves, and the linear
term contributes nothing, so $h$ is $(1+\rho)$-strongly convex relative to
$\Phi$: for all $u\in\simplex$ and $v$ in the relative interior,
\begin{equation}\label{eq:h-relsc}
h(u)=h(v)+\ip{\nabla h(v)}{u-v}+(1+\rho)\KL(u\|v).
\end{equation}
The function $h$ is differentiable at $w_{t+1}$: if $\kappa>0$, every point of
$C$ has all coordinates at least $\kappa>0$, so $h$ is differentiable on all of
$C$, including at points with some $w_i=\kappa$ (where the clipped solution may
well lie); if $\kappa=0$, the minimizer has all coordinates positive because
$\partial_i\KL(w\|w_t)=\ln(w_i/w_{t,i})+1\to-\infty$ as $w_i\downarrow0$ while
$w_t$ has positive coordinates. Hence first-order optimality over the convex
set $C$ gives
$\ip{\nabla h(w_{t+1})}{u-w_{t+1}}\ge0$ for all $u\in C$. Hence, by
\eqref{eq:h-relsc} at $v=w_{t+1}$,
\[
h(u)\ge h(w_{t+1})+(1+\rho)\KL(u\|w_{t+1})\qquad(u\in C).
\]
Expanding $h$ on both sides and rearranging,
\begin{equation}\label{eq:md1}
\eta\ip{g_t}{u-w_{t+1}}\le\KL(u\|w_t)-(1+\rho)\KL(u\|w_{t+1})-\KL(w_{t+1}\|w_t)
+\rho\big(r(u)-r(w_{t+1})\big),
\end{equation}
with equality when $\kappa=0$, since then $C=\simplex$ and
$\nabla h(w_{t+1})$ is a multiple of $\mathbf 1$, orthogonal to $u-w_{t+1}$.
Next, split $\ip{g_t}{u-w_t}=\ip{g_t}{u-w_{t+1}}+\ip{g_t}{w_{t+1}-w_t}$ and
bound the second term by H\"older, Young's inequality
$ab\le\tfrac12a^2+\tfrac12b^2$ and Pinsker's inequality
$\|p-q\|_1^2\le2\KL(p\|q)$:
\begin{equation}\label{eq:md2}
\eta\ip{g_t}{w_{t+1}-w_t}\le\eta\|g_t\|_\infty\|w_{t+1}-w_t\|_1
\le\tfrac12\eta^2\|g_t\|_\infty^2+\tfrac12\|w_{t+1}-w_t\|_1^2
\le\tfrac12\eta^2\|g_t\|_\infty^2+\KL(w_{t+1}\|w_t).
\end{equation}
Adding \eqref{eq:md1} and \eqref{eq:md2}, the terms $\pm\KL(w_{t+1}\|w_t)$
cancel and we obtain \eqref{eq:md}. The claims about $u=u_K$ and general $u$
follow from $r(u_K)=0$, $r\ge0$ and $r\le\ln K$ on $\simplex$. At
$\tau=\kappa=0$ we have $\rho=0$ and $C=\simplex$, and \eqref{eq:md1} is the
three-point identity
$\eta\ip{g_t}{u-w_{t+1}}=\KL(u\|w_t)-\KL(u\|w_{t+1})-\KL(w_{t+1}\|w_t)$ of
exponentiated gradient; the iterates then stay in the relative interior since
$w_0$ is uniform and the update multiplies by positive factors, so
\eqref{eq:md} holds for every $u\in\simplex$, including the vertices $u=e_i$
used in Lemma~\ref{lem:M}.
\end{proof}

\begin{proof}[Proof of Lemma~\ref{lem:M}]
Set $\bar\gamma:=\gamma R=\sqrt{2\Delta/(LT)}$ and
$\bar\eta:=\eta R=\sqrt{2\ln K}/(M\sqrt T)$, and define
\[
\mathcal C \;:=\; \frac{\bar\gamma\ln K}{\bar\eta}
\;+\;\frac{T\bar\gamma\bar\eta M^2}{2}\;+\;\frac{TL\bar\gamma^2}{2}.
\]
Substituting the step sizes,
\begin{equation}\label{eq:Ccalc}
\frac{\bar\gamma\ln K}{\bar\eta}=M\sqrt{\tfrac{\Delta\ln K}{L}},\qquad
\frac{T\bar\gamma\bar\eta M^2}{2}=M\sqrt{\tfrac{\Delta\ln K}{L}},\qquad
\frac{TL\bar\gamma^2}{2}=\Delta,
\end{equation}
so $\mathcal C=2M\sqrt{\Delta\ln K/L}+\Delta$.

We prove by induction on $t$ that $\dn{\nabla f_i(x_t)}\le M$ for all $i$.

\emph{Base case.} By Assumption~\ref{as:lb} and Lemma~\ref{lem:gradbound},
$\dn{\nabla f_i(x_0)}\le\sqrt{2L\Delta}$. Since $K\ge2$ we have
$4\ln K\ge4\ln2>2$, so $2\sqrt{L\Delta\ln K}\ge\sqrt{2L\Delta}$ and the
hypothesis on $M$ gives $\dn{\nabla f_i(x_0)}\le M$.

\emph{Inductive step.} Suppose $\dn{\nabla f_i(x_s)}\le M$ for all $i$ and all
$s\le t$. Then $\|g_s\|_\infty\le\max_i\pn{D_s}\dn{\nabla f_i(x_s)}\le RM$ for
$s\le t$ by H\"older. Apply Lemma~\ref{lem:onestep} with $w=e_i$ and add
$\gamma/\eta$ times \eqref{eq:md} with $u=e_i$; summing over $s=0,\dots,t$ and
using $\KL(e_i\|w_0)=\ln K$, $\KL(e_i\|w_{t+1})\ge0$ and $t+1\le T$,
\[
f_i(x_{t+1})-f_i(x_0)\;\le\;-\bar\gamma\sum_{s=0}^{t}\dn{G_s}
+\frac{\bar\gamma\ln K}{\bar\eta}+\frac{T\bar\gamma\bar\eta M^2}{2}
+\frac{TL\bar\gamma^2}{2}\;\le\;\mathcal C,
\]
where we dropped the non-positive first sum. Hence
$f_i(x_{t+1})-f_i^\star\le\Delta+\mathcal C=2\Delta+2M\sqrt{\Delta\ln K/L}$, and
Lemma~\ref{lem:gradbound} gives
\[
\dn{\nabla f_i(x_{t+1})}^2\;\le\;2L\Big(2\Delta+2M\sqrt{\tfrac{\Delta\ln K}{L}}\Big)
\;=\;4L\Delta+4M\sqrt{L\Delta\ln K}.
\]
It therefore suffices that $M^2\ge4L\Delta+4M\sqrt{L\Delta\ln K}$, i.e. that $M$
be at least the positive root of $M^2-4\sqrt{L\Delta\ln K}\,M-4L\Delta=0$, namely
\[
M\;\ge\;2\sqrt{L\Delta\ln K}+\sqrt{4L\Delta\ln K+4L\Delta}
\;=\;2\sqrt{L\Delta\ln K}+2\sqrt{L\Delta(\ln K+1)},
\]
which is the hypothesis. This closes the induction and shows
$M=O(\sqrt{L\Delta\ln K})$ suffices. Note that $M$ is defined by an explicit
formula in $L,\Delta,K$ only, so the step sizes are well defined and there is
no circularity.
\end{proof}

\begin{proof}[Proof of Theorem~\ref{thm:det}]
Take the uniform comparator $u=u_K$, which lies in $\simplex^\kappa$ for every
$\kappa\le1/K$. Adding Lemma~\ref{lem:onestep} with $w=u_K$ to $\gamma/\eta$
times \eqref{eq:md} with $u=u_K$ gives, for each $t$, using $r(u_K)=0$,
$r(w_{t+1})\ge0$ and $(1+\rho)\KL\ge\KL$,
\[
F(x_{t+1})u_K\le F(x_t)u_K-\gamma R\dn{G_t}
+\frac\gamma\eta\Big(\KL(u_K\|w_t)-\KL(u_K\|w_{t+1})\Big)
+\frac{\gamma\eta}{2}\|g_t\|_\infty^2+\frac{L\gamma^2R^2}{2}.
\]
By hypothesis (or by Lemma~\ref{lem:M} when $\tau=\kappa=0$),
$\|g_t\|_\infty\le RM$ for all $t<T$. Summing over $t=0,\dots,T-1$,
telescoping the divergences, and using $\KL(u_K\|w_0)=0$ (as $w_0=u_K$) and
$\KL(u_K\|w_T)\ge0$,
\[
F(x_T)u_K-F(x_0)u_K\le-\gamma R\sum_{t=0}^{T-1}\dn{G_t}
+\frac{T\gamma\eta R^2M^2}{2}+\frac{TL\gamma^2R^2}{2}.
\]
Moreover $F(x_0)u_K-F(x_T)u_K\le\frac1K\sum_i(f_i(x_0)-f_i^\star)\le\Delta$ by
Assumption~\ref{as:lb}. Rearranging and dividing by $\gamma RT$ gives
\eqref{eq:det-master}; neither $\tau$, $\kappa$ nor $K$ has entered. With
$\bar\gamma=\gamma R=\sqrt{2\Delta/(LT)}$ the first and third terms of
\eqref{eq:det-master} are each $\sqrt{\Delta L/(2T)}$, summing to
$\sqrt{2\Delta L/T}$; with $\bar\eta=\eta R=\sqrt{2\ln K}/(M\sqrt T)$ the
second is $M\sqrt{\ln K/(2T)}$. Adding gives \eqref{eq:det}. The right-hand
side of \eqref{eq:det} is at most $\epsilon$ once
$T\ge(\sqrt{2\Delta L}+M\sqrt{\ln K/2})^2/\epsilon^2$, and
$(\sqrt{2\Delta L}+M\sqrt{\ln K/2})^2\le4\Delta L+M^2\ln K$, so
$T=\lceil(4\Delta L+M^2\ln K)/\epsilon^2\rceil$ iterations suffice. Every step
used only Assumptions~\ref{as:smooth}--\ref{as:lb} and the identity
\eqref{eq:lmo}; no property of $\pn\cdot$ beyond being a norm was invoked.
\end{proof}

\begin{corollary}[Average CA distance]\label{cor:ca}
Under the conditions of \eqref{eq:det},
$\tfrac1T\sum_{t}\mathrm{CA}_t\le 2\big(\sqrt{2\Delta L}+M\sqrt{\log K/2}\big)/\sqrt T$
for every measurable selection $w^\star(x_t)$.
\end{corollary}

\begin{proof}[Proof of Corollary~\ref{cor:ca}]
Let $w^\star(x_t)$ be any minimizer of $w\mapsto\dn{\gF(x_t)w}$ over $\simplex$.
By the triangle inequality and then optimality,
\[
\mathrm{CA}_t=\dn{\gF(x_t)w_t-\gF(x_t)w^\star(x_t)}
\le\dn{\gF(x_t)w_t}+\dn{\gF(x_t)w^\star(x_t)}\le2\dn{\gF(x_t)w_t}.
\]
Averaging over $t$ and applying Theorem~\ref{thm:det} gives the claim. Only the
triangle inequality and the optimality of $w^\star$ were used, so the bound is
valid simultaneously for every measurable selection; in particular it does not
presuppose the uniqueness of $\gF(x_t)w^\star(x_t)$ discussed in
Proposition~\ref{prop:unique}.
\end{proof}

\begin{remark}\label{rem:hilbert}
In a Hilbert dual norm the constant $2$ can be removed: the first-order
optimality condition for $w^\star$ gives
$\ip{\gF w^\star}{\gF w_t-\gF w^\star}\ge0$, whence
$\dn{\gF w_t-\gF w^\star}^2\le\dn{\gF w_t}^2-\dn{\gF w^\star}^2\le\dn{\gF w_t}^2$.
This is the argument used in the Euclidean literature. It relies on the
parallelogram identity and is unavailable for a general dual norm, which is why
we state Corollary~\ref{cor:ca} with the triangle inequality instead. The same
identity gives the conversion used in Corollary~\ref{cor:gap}: writing
$a=\dn{\gF w_t}$ and $b=\dn{\gF w^\star}$,
\[
\mathrm{CA}_t^2\le a^2-b^2=(a+b)(a-b)\le2M\,\mathrm{gap}_t ,
\]
since $a,b\le M$ along the trajectory (Lemma~\ref{lem:M}) and
$a-b=\mathrm{gap}_t$. More generally, if the dual norm satisfies \eqref{eq:S2}
with power type $q$, then with $w^\star$ minimizing $\varphi_q(\gF w)$ over
$\simplex$ (the same minimizers as $\dn{\gF w}$), first-order optimality and
\eqref{eq:S2} give $\tfrac{\mu}{q}\mathrm{CA}_t^q\le\varphi_q(\gF w_t)-\varphi_q(\gF w^\star)
=\tfrac1q(a^q-b^q)\le a^{q-1}(a-b)\le M^{q-1}\mathrm{gap}_t$, i.e.
$\mathrm{CA}_t^q\le(q/\mu)M^{q-1}\mathrm{gap}_t$.
The direct inequalities $\mathrm{gap}_t\le\dn{\gF(x_t)w_t}$ and
$\mathrm{CA}_t\le2\dn{\gF(x_t)w_t}$ imply average-gap and average-CA
convergence whenever the actual blended-gradient residual converges in
average. This includes a fixed scalarization on its own convergent trajectory.
Corollary~\ref{cor:gap} gives an additional comparator-dependent decomposition;
it is not needed to infer average-gap convergence from Theorem~\ref{thm:det}.
\end{remark}

\begin{proof}[Proof of Theorem~\ref{thm:regret}]
Fix $u\in\simplex$ and write $\psi_t(w):=\dn{A_tw}$. The support-function
argument of Lemma~\ref{lem:subgrad} applies to any linear map $A_t$, with
$d=D_t\in\lmo(G_t)$ and $G_t=A_tw_t$. Thus $-g_t/R$ is a subgradient of
$\psi_t$ at $w_t$.
The subgradient inequality $\psi_t(u)\ge\psi_t(w_t)+\ip{-g_t/R}{u-w_t}$
rearranges to
\begin{equation}\label{eq:subgrad-step}
\psi_t(w_t)-\psi_t(u)\le\frac1R\ip{g_t}{u-w_t}\qquad\text{for every }t .
\end{equation}
Summing \eqref{eq:md} at $\tau=\kappa=0$ over $t=0,\dots,T-1$, telescoping and dropping
$-\KL(u\|w_T)\le0$,
\[
\eta\sum_{t=0}^{T-1}\ip{g_t}{u-w_t}\le\KL(u\|w_0)+\frac{\eta^2}{2}\sum_{t=0}^{T-1}\|g_t\|_\infty^2
\le\ln K+\frac{\eta^2TR^2M^2}{2},
\]
using $\KL(u\|w_0)\le\ln K$ for uniform $w_0$ and $\|g_t\|_\infty\le RM$ from
the column bound $\max_i\dn{A_te_i}\le M$ and H\"older's inequality.
Combining with \eqref{eq:subgrad-step} and dividing by $RT$,
\[
\frac1T\sum_{t=0}^{T-1}\big[\psi_t(w_t)-\psi_t(u)\big]
\le\frac{\ln K}{\eta RT}+\frac{\eta RM^2}{2}
=M\sqrt{\frac{2\ln K}{T}},
\]
the last equality by the choice $\eta R=\sqrt{2\ln K}/(M\sqrt T)$, as in the
proof of Theorem~\ref{thm:det}. This is \eqref{eq:regret}. Since it holds for
every $u\in\simplex$, it holds in particular for a minimizer of
$u\mapsto\frac1T\sum_t\psi_t(u)$ (which exists: each $\psi_t$ is continuous and
$\simplex$ is compact), the best fixed weighting in hindsight. The argument used
Lemma~\ref{lem:subgrad}, Lemma~\ref{lem:md} and the bound $\|g_t\|_\infty\le RM$
only. Hence it holds for arbitrary bounded matrix sequences, including
sequences depending on preceding iterates; no outer progress is needed.
\end{proof}

\begin{proof}[Proof of Corollary~\ref{cor:gap}]
Write $\psi_t(w)=\dn{\gF(x_t)w}$. For any $u\in\simplex$ and any $t$, by
definition of the gap,
\[
\mathrm{gap}_t=\big[\psi_t(w_t)-\psi_t(u)\big]+\Big[\psi_t(u)-\min_{w\in\simplex}\psi_t(w)\Big].
\]
Average over $t$ and apply \eqref{eq:regret} to the first bracket:
\[
\frac1T\sum_t\mathrm{gap}_t\le M\sqrt{\frac{2\ln K}{T}}
+\frac1T\sum_t\Big[\psi_t(u)-\min_{w}\psi_t(w)\Big]\qquad\text{for every }u\in\simplex .
\]
The second sum is a continuous function of $u$ on the compact set $\simplex$
and is non-negative termwise; minimizing over $u$ gives $V_T$ and hence
\eqref{eq:gap}. If some $\bar u$ minimizes $\psi_t$ for every $t$ then each
bracket vanishes at $u=\bar u$ and $V_T=0$. The Euclidean conversion
$\mathrm{CA}_t^2\le2M\,\mathrm{gap}_t$ is the remark preceding this proof;
averaging and using the concavity of $\sqrt{\cdot}$ (Jensen),
$\frac1T\sum_t\mathrm{CA}_t\le\frac1T\sum_t\sqrt{2M\,\mathrm{gap}_t}
\le\sqrt{2M\cdot\frac1T\sum_t\mathrm{gap}_t}$.
\end{proof}

\section{Proofs for the stochastic setting}\label{app:stoch}

\begin{algorithm}[t]
\caption{Entropic LMO-MGDA with momentum on the blended gradient ($\tau=\kappa=0$)}
\label{alg:stoch}
\begin{algorithmic}[1]
\Require $x_0$; step sizes $\gamma,\eta>0$; momentum $\alpha\in(0,1]$; radius $R$; horizon $T$
\State $w_0\gets(1/K,\dots,1/K)$;\quad $\mathcal M_{-1}\gets 0$
\For{$t=0,1,\dots,T-1$}
  \State $D_t\gets\lmo(\mathcal M_{t-1})$
    \Comment{oracle at the \emph{previous} buffer: $D_t$ is $\mathcal F_t$-measurable}
  \State sample $\xi_{t,i}$ and compute $\nabla f_i(x_t;\xi_{t,i})$ for $i\in[K]$
  \State $g_{t,i}\gets\ip{D_t}{\nabla f_i(x_t;\xi_{t,i})}$ for $i\in[K]$
  \State $w_{t+1,i}\gets w_{t,i}e^{\eta g_{t,i}}\big/\sum_jw_{t,j}e^{\eta g_{t,j}}$
    \Comment{line~6 of Algorithm~\ref{alg:main} at $\tau=\kappa=0$}
  \State $\widehat G_t\gets\sum_iw_{t,i}\nabla f_i(x_t;\xi_{t,i})$;\quad
         $\mathcal M_t\gets(1-\alpha_t)\mathcal M_{t-1}+\alpha_t\widehat G_t$,
         \ \ $\alpha_0=1$, $\alpha_t=\alpha$ for $t\ge1$
  \State $x_{t+1}\gets x_t+\gamma D_t$
\EndFor
\end{algorithmic}
\end{algorithm}

Only \emph{one} vector of dimension $d$ is carried between iterations. The
convention $\alpha_0=1$ removes the initialization bias of the moving average.
At $t=0$ the oracle is evaluated at $\mathcal M_{-1}=0$; with the convention
$\lmo(0)=0$ of Section~\ref{sec:setup} the first step is $D_0=0$, and the
argument below in fact allows any selection in $B$.

\paragraph{Filtration and notation.}
Let $\mathcal F_t$ be the $\sigma$-algebra generated by
$\xi_{0,\cdot},\dots,\xi_{t-1,\cdot}$ (and $x_0$), so that
$x_t$, $w_t$, $\mathcal M_{t-1}$ and hence $D_t$ are $\mathcal F_t$-measurable
while the $\xi_{t,i}$ are not. Write $A_t:=\gF(x_t)$ and
\[
\lambda_t:=\mathcal M_t-A_tw_t,\qquad
n_t:=\widehat G_t-A_tw_t=\sum_iw_{t,i}n_{t,i},\qquad
n_{t,i}:=\nabla f_i(x_t;\xi_{t,i})-\nabla f_i(x_t),
\]
together with the conventions $\lambda_{-1}:=\mathcal M_{-1}-A_0w_0=-A_0w_0$
and $\Theta_0:=0$, where for $t\ge1$
\[
\Theta_t:=A_{t-1}w_{t-1}-A_tw_t .
\]
By Assumption~\ref{as:noise}, $\E[n_{t,i}\mid\mathcal F_t]=0$ and
$\E[\|n_{t,i}\|_2^2\mid\mathcal F_t]\le\sigma^2$; consequently
$\E[n_t\mid\mathcal F_t]=0$. Since $w_t\in\simplex$ is
$\mathcal F_t$-measurable, the triangle inequality and Jensen's inequality for
the convex map $s\mapsto s^2$ with weights $w_t$ give
\begin{equation}\label{eq:nfree}
\E\big[\|n_t\|_2^2\mid\mathcal F_t\big]
\le\E\Big[\Big(\sum_iw_{t,i}\|n_{t,i}\|_2\Big)^2\Big|\,\mathcal F_t\Big]
\le\sum_iw_{t,i}\,\E\big[\|n_{t,i}\|_2^2\mid\mathcal F_t\big]\le\sigma^2 .
\end{equation}
This bound is free of $K$ and, notably, uses no independence across tasks: a
shared minibatch for all $K$ objectives is allowed. Throughout this appendix we
work under Assumption~\ref{as:gradbound}, so that
$\max_i\dn{\nabla f_i(x_t)}\le M$ for every $t$, and we write
$\zeta=\max_{x\ne0}\dn x/\|x\|_2$ as in Section~\ref{sec:setup}.

\subsection{The weight increment}

The tracking analysis needs to control how far the weights move in one step.
The pathwise bound (a) below is what the deterministic analysis would use; it
involves $\|g_t\|_\infty$, a maximum over tasks, which in the stochastic setting
contains $\max_i\dn{n_{t,i}}$ and would introduce a $\sqrt K$ through a maximal
inequality. Part (b) shows that \emph{in expectation} the increment is governed
by the $w_t$-\emph{weighted average} of the noise instead, with the maximum over
tasks confined to a term of second order in $\eta$.

\begin{lemma}[Weight increment; restatement of Lemma~\ref{lem:wstep}]\label{lem:wstep-app}
Let $w\in\simplex$, $g\in\R^K$, and $w'_i=w_ie^{\eta g_i}/\sum_jw_je^{\eta g_j}$.
\begin{enumerate}[leftmargin=1.8em]
\item[(a)] Pathwise, $\|w'-w\|_1\le2\eta\|g\|_\infty$.
\item[(b)] Suppose $w$, $x$ and $D$ with $\pn D\le R$ are $\mathcal F$-measurable,
$g_i=\ip{D}{\nabla f_i(x)}+\ip{D}{n_i}$ with $\max_i\dn{\nabla f_i(x)}\le M$,
$\E[n_i\mid\mathcal F]=0$ and $\E[\|n_i\|_2^2\mid\mathcal F]\le\sigma^2$ for
each $i$ (no independence across $i$ required). If $\eta RM\le\tfrac12$, then
\begin{equation}\label{eq:wstep-exp}
\E\big[\|w'-w\|_1\mid\mathcal F\big]\;\le\;\omega
\;:=\;16\,\eta R\,(M+\zeta\sigma)\;+\;8K\eta^2R^2\zeta^2\sigma^2 .
\end{equation}
\end{enumerate}
\end{lemma}

\begin{proof}
(a) Take $u=w$ in \eqref{eq:md1}, which at $\tau=\kappa=0$ is the three-point
identity of exponentiated gradient (with $w_t=w$, $w_{t+1}=w'$, $g_t=g$). Since
$\KL(w\|w)=0$ it reads
$\eta\ip{g}{w-w'}=-\KL(w\|w')-\KL(w'\|w)$, i.e.
\[
\KL(w'\|w)+\KL(w\|w')=\eta\ip{g}{w'-w}\le\eta\|g\|_\infty\|w'-w\|_1 .
\]
Both divergences are non-negative, so
$\KL(w'\|w)\le\eta\|g\|_\infty\|w'-w\|_1$. Pinsker's inequality
$\tfrac12\|w'-w\|_1^2\le\KL(w'\|w)$ gives
$\tfrac12\|w'-w\|_1^2\le\eta\|g\|_\infty\|w'-w\|_1$; if $w'\ne w$ divide by
$\|w'-w\|_1$, and otherwise the claim is trivial.

(b) \emph{Step 1: a weighted-average bound.} Put $b_i:=\eta g_i$,
$a_i:=e^{b_i}$, $Z:=\sum_jw_ja_j$ and $\bar b:=\sum_jw_jb_j$. Then
$w'_i-w_i=w_i(a_i-Z)/Z$, so $\|w'-w\|_1=Z^{-1}\sum_iw_i|a_i-Z|$. For any
constant $c$, since $\sum_iw_i(a_i-c)=Z-c$,
\[
\sum_iw_i|a_i-Z|\le\sum_iw_i|a_i-c|+|c-Z|
=\sum_iw_i|a_i-c|+\Big|\sum_iw_i(c-a_i)\Big|\le2\sum_iw_i|a_i-c| .
\]
Take $c=e^{\bar b}$. By Jensen's inequality for $\exp$,
$Z=\sum_jw_je^{b_j}\ge e^{\bar b}$, hence $e^{\bar b}/Z\le1$ and
\begin{equation}\label{eq:wavg}
\|w'-w\|_1\;\le\;\frac{2e^{\bar b}}{Z}\sum_iw_i\big|e^{b_i-\bar b}-1\big|
\;\le\;2\sum_iw_i\,\big|e^{y_i}-1\big|,\qquad y_i:=b_i-\bar b .
\end{equation}

\emph{Step 2: splitting the scores.} Write $b_i=d_i+\nu_i$ with
$d_i:=\eta\ip{D}{\nabla f_i(x)}$ and $\nu_i:=\eta\ip{D}{n_i}$. By H\"older,
$|d_i|\le\eta RM$ and $|\nu_i|\le\eta R\dn{n_i}\le\eta R\zeta\|n_i\|_2$. Hence
\[
|y_i|\le|d_i-\bar d|+|\nu_i-\bar\nu|\le2\eta RM+|\nu_i|+\sum_jw_j|\nu_j| ,
\]
and consequently
$\sum_iw_i|y_i|\le2\eta RM+2\sum_iw_i|\nu_i|$.

\emph{Step 3: the good event.} Let
$\mathcal A:=\{\max_j|\nu_j|\le\tfrac12\}$. On $\mathcal A$, using
$2\eta RM\le1$, every $|y_i|\le1+\tfrac12+\tfrac12=2$. Since
$y\mapsto(e^{y}-1)/y$ is increasing on $(0,2]$ with value $(e^2-1)/2<4$ at
$y=2$, we have $|e^{y}-1|\le e^{|y|}-1\le4|y|$ for $|y|\le2$, and
\eqref{eq:wavg} gives, on $\mathcal A$,
\[
\|w'-w\|_1\le8\sum_iw_i|y_i|\le16\eta RM+16\sum_iw_i|\nu_i| .
\]
Off $\mathcal A$ we use the trivial bound $\|w'-w\|_1\le2$, together with the
union bound and Chebyshev's inequality:
\[
\Pr(\mathcal A^c\mid\mathcal F)\le\sum_j\Pr\big(|\nu_j|>\tfrac12\,\big|\,\mathcal F\big)
\le4\sum_j\E[\nu_j^2\mid\mathcal F]
\le4\eta^2R^2\zeta^2\sum_j\E[\|n_j\|_2^2\mid\mathcal F]\le4K\eta^2R^2\zeta^2\sigma^2 .
\]
\emph{Step 4.} Combining, and using
$\E[|\nu_i|\mid\mathcal F]\le\eta R\zeta\,\E[\|n_i\|_2\mid\mathcal F]\le\eta R\zeta\sigma$
(Jensen),
\[
\E\big[\|w'-w\|_1\mid\mathcal F\big]
\le16\eta RM+16\eta R\zeta\sigma\sum_iw_i+2\Pr(\mathcal A^c\mid\mathcal F)
\le16\eta R(M+\zeta\sigma)+8K\eta^2R^2\zeta^2\sigma^2 . \qedhere
\]
\end{proof}

The constant $16$ is not optimized. What matters is the structure of
\eqref{eq:wstep-exp}: the first-order term is a $w$-weighted average of the
per-task noise, not a maximum, and the number of objectives enters only through
the second-order term $8K\eta^2R^2\zeta^2\sigma^2$, which is what makes the
leading term of Theorem~\ref{thm:stoch} free of $K$ apart from $\log K$.

\subsection{Tracking error}

\begin{proof}[Proof of Lemma~\ref{lem:track}]
From line~7, for $t\ge1$,
\begin{align*}
\lambda_t&=(1-\alpha)\mathcal M_{t-1}+\alpha\widehat G_t-A_tw_t\\
&=(1-\alpha)\big(\mathcal M_{t-1}-A_{t-1}w_{t-1}\big)
+(1-\alpha)\big(A_{t-1}w_{t-1}-A_tw_t\big)+\alpha n_t\\
&=(1-\alpha)\lambda_{t-1}+(1-\alpha)\Theta_t+\alpha n_t ,
\end{align*}
and $\lambda_0=\widehat G_0-A_0w_0=n_0$ because $\alpha_0=1$. Unrolling,
\begin{equation}\label{eq:split}
\lambda_t=\underbrace{\alpha\sum_{s=1}^t(1-\alpha)^{t-s}n_s+(1-\alpha)^tn_0}_{=:\ \nu_t}
\;+\;\underbrace{\sum_{s=1}^t(1-\alpha)^{t-s+1}\Theta_s}_{=:\ \delta_t}.
\end{equation}

\emph{(i) Drift.} Split $\Theta_s$ according to whether $x$ or $w$ moved:
\[
\Theta_s=\big[A_{s-1}-A_s\big]w_{s-1}\;+\;A_s\big(w_{s-1}-w_s\big).
\]
For the first term, by convexity of $\dn\cdot$ and Assumption~\ref{as:smooth},
\[
\dn{\big[A_{s-1}-A_s\big]w_{s-1}}
\le\sum_iw_{s-1,i}\dn{\nabla f_i(x_{s-1})-\nabla f_i(x_s)}
\le L\pn{x_s-x_{s-1}}\le L\gamma R,
\]
the last step because $x_s-x_{s-1}=\gamma D_{s-1}$ with $\pn{D_{s-1}}\le R$;
this holds \emph{pathwise}, for every realization of the noise, and is free
of $K$. For the second term, by convexity and Assumption~\ref{as:gradbound},
\[
\dn{A_s(w_{s-1}-w_s)}\le\max_i\dn{\nabla f_i(x_s)}\,\|w_s-w_{s-1}\|_1
\le M\|w_s-w_{s-1}\|_1 .
\]
The step $w_{s-1}\to w_s$ is the update of Lemma~\ref{lem:wstep} with
$\mathcal F=\mathcal F_{s-1}$, $w=w_{s-1}$, $x=x_{s-1}$, $D=D_{s-1}$ and noise
$n_i=n_{s-1,i}$; the hypotheses of part (b) hold by Assumptions~\ref{as:noise}
and~\ref{as:gradbound} and $\eta RM\le\frac12$. Hence
$\E\|w_s-w_{s-1}\|_1\le\omega$, and
\[
\E\dn{\Theta_s}\le L\gamma R+M\omega=:B\qquad(s\ge1),
\]
while $\Theta_0=0$. Summing the geometric weights in \eqref{eq:split},
\[
\E\dn{\delta_t}\le B\sum_{s=1}^t(1-\alpha)^{t-s+1}\le B\,\frac{1-\alpha}{\alpha}\le\frac{B}{\alpha}.
\]

\emph{(ii) Noise.} For $s<s'$, $n_s$ is $\mathcal F_{s'}$-measurable and
$\E[n_{s'}\mid\mathcal F_{s'}]=0$, so $\E\ip{n_s}{n_{s'}}=0$: the summands of
$\nu_t$ are orthogonal in $L^2$. Therefore, by \eqref{eq:nfree},
\begin{align*}
\E\|\nu_t\|_2^2
&=\alpha^2\sum_{s=1}^t(1-\alpha)^{2(t-s)}\E\|n_s\|_2^2+(1-\alpha)^{2t}\E\|n_0\|_2^2\\
&\le\sigma^2\Big(\frac{\alpha^2}{1-(1-\alpha)^2}+(1-\alpha)^{2t}\Big)
\le\sigma^2\big(\alpha+(1-\alpha)^{2t}\big),
\end{align*}
using $\alpha^2/(1-(1-\alpha)^2)=\alpha/(2-\alpha)\le\alpha$ for
$\alpha\in(0,1]$. Converting with $\zeta$ and applying Jensen's inequality
($\E Y\le\sqrt{\E Y^2}$) and $\sqrt{a+b}\le\sqrt a+\sqrt b$,
\[
\E\dn{\nu_t}\le\zeta\sqrt{\E\|\nu_t\|_2^2}\le\zeta\sigma\big(\sqrt\alpha+(1-\alpha)^t\big).
\]
Combining (i) and (ii) with $\dn{\lambda_t}\le\dn{\nu_t}+\dn{\delta_t}$ gives
\begin{equation}\label{eq:lambda-bound}
\E\dn{\lambda_t}\le\zeta\sigma\big(\sqrt\alpha+(1-\alpha)^t\big)+\frac{L\gamma R+M\omega}{\alpha}
\qquad(t\ge0),
\end{equation}
which is the statement of Lemma~\ref{lem:track}. No maximal inequality over the
$K$ tasks has been used; $K$ enters only through the second-order term of
$\omega$.
\end{proof}

\subsection{One-step inequality and the rate}

\begin{lemma}[Stochastic one-step inequality]\label{lem:onestep-stoch}
Under Assumption~\ref{as:smooth}, for every $w\in\simplex$ and every $t\ge0$,
\begin{align*}
F(x_{t+1})w\le F(x_t)w&-\gamma R\dn{A_tw_t}
+\gamma\ip{g_t}{w-w_t}
+2\gamma R\big(\dn{\lambda_{t-1}}+\dn{\Theta_t}\big)\\
&-\gamma\Big\langle\textstyle\sum_i(w_i-w_{t,i})\,n_{t,i},\,D_t\Big\rangle
+\tfrac12L\gamma^2R^2 ,
\end{align*}
with the conventions $\lambda_{-1}=-A_0w_0$, $\Theta_0=0$. The noise term is
kept \emph{signed}: it is bounded in absolute value by
$\gamma R\dn{\sum_i(w_i-w_{t,i})n_{t,i}}$, but what we use is that its
conditional mean given $\mathcal F_t$ is zero when $w$ is deterministic.
\end{lemma}

\begin{proof}
Weighting Lemma~\ref{lem:descent} by $w$ exactly as in the proof of
Lemma~\ref{lem:onestep} gives
$F(x_{t+1})w\le F(x_t)w+\gamma\ip{A_tw}{D_t}+\tfrac12L\gamma^2R^2$. Write
$A_tw=A_tw_t+A_t(w-w_t)$.

\emph{First part.} For $t\ge1$, $D_t\in\lmo(\mathcal M_{t-1})$, so by
\eqref{eq:lmo} and H\"older,
\[
\ip{A_tw_t}{D_t}=\ip{\mathcal M_{t-1}}{D_t}+\ip{A_tw_t-\mathcal M_{t-1}}{D_t}
\le-R\dn{\mathcal M_{t-1}}+R\dn{A_tw_t-\mathcal M_{t-1}} .
\]
By the triangle inequality
$\dn{\mathcal M_{t-1}}\ge\dn{A_tw_t}-\dn{A_tw_t-\mathcal M_{t-1}}$, hence
$\ip{A_tw_t}{D_t}\le-R\dn{A_tw_t}+2R\dn{A_tw_t-\mathcal M_{t-1}}$. Finally
$A_tw_t-\mathcal M_{t-1}=-(\mathcal M_{t-1}-A_{t-1}w_{t-1})-(A_{t-1}w_{t-1}-A_tw_t)
=-\lambda_{t-1}-\Theta_t$, and the triangle inequality gives the term
$2R(\dn{\lambda_{t-1}}+\dn{\Theta_t})$. For $t=0$, $\mathcal M_{-1}=0$ and
$D_0\in B$ is arbitrary, so H\"older gives
$\ip{A_0w_0}{D_0}\le R\dn{A_0w_0}=-R\dn{A_0w_0}+2R\dn{A_0w_0}
=-R\dn{A_0w_0}+2R(\dn{\lambda_{-1}}+\dn{\Theta_0})$ under the stated
conventions.

\emph{Second part.} By definition,
$g_{t,i}=\ip{D_t}{\nabla f_i(x_t;\xi_{t,i})}=\ip{D_t}{\nabla f_i(x_t)}+\ip{D_t}{n_{t,i}}$,
so
\[
\ip{A_t(w-w_t)}{D_t}=\sum_i(w_i-w_{t,i})\ip{\nabla f_i(x_t)}{D_t}
=\ip{g_t}{w-w_t}-\Big\langle\sum_i(w_i-w_{t,i})\,n_{t,i},\,D_t\Big\rangle .
\]
Multiplying both parts by $\gamma$ and collecting terms gives the claim.
\end{proof}

\begin{proof}[Proof of Theorem~\ref{thm:stoch}]
Throughout, $\alpha=T^{-1/2}$, $\gamma R=T^{-3/4}$ and $\eta R=T^{-3/4}/M^2$;
the hypothesis $T\ge(2/M)^{4/3}$ is exactly $\eta RM\le\frac12$, so
Lemma~\ref{lem:wstep-app}(b), and with it Lemma~\ref{lem:track}, apply.

\emph{Step 1: the master inequality.} Fix a deterministic $u\in\simplex$. Add
Lemma~\ref{lem:onestep-stoch} with $w=u$ to $\gamma/\eta$ times \eqref{eq:md}
(at $\tau=\kappa=0$) with the same $u$; Lemma~\ref{lem:md} is a deterministic
statement about line~6 of Algorithm~\ref{alg:stoch} and applies verbatim to
the noisy scores, since its proof uses no property of $g_t$ beyond finiteness.
Summing over $t=0,\dots,T-1$, telescoping the divergences, using
$\KL(u\|w_T)\ge0$ and $F(x_0)u-F(x_T)u\le\Delta$ (Assumption~\ref{as:lb}), and
dividing by $\gamma RT$,
\begin{align}\label{eq:master-stoch}
\frac1T\sum_{t=0}^{T-1}\dn{A_tw_t}
\le\ &\frac{\Delta}{\gamma RT}+\frac{\KL(u\|w_0)}{\eta RT}
+\frac{\eta}{2RT}\sum_{t}\|g_t\|_\infty^2
+\frac{L\gamma R}{2}\notag\\
&+\frac2T\sum_{t=0}^{T-1}\big(\dn{\lambda_{t-1}}+\dn{\Theta_t}\big)
-\frac{1}{RT}\sum_{t=0}^{T-1}\Big\langle\textstyle\sum_i(u_i-w_{t,i})\,n_{t,i},\,D_t\Big\rangle .
\end{align}
This holds pathwise. For the stationarity bound we take $u=u_K$, so that
$\KL(u\|w_0)=0$ and the second term vanishes; for the regret bound of
Corollary~\ref{cor:regret-stoch} a general $u$ is kept and
$\KL(u\|w_0)\le\ln K$. We now take expectations term by term.

\emph{Step 2: the noise cross term vanishes.} $u$ is deterministic and $w_t$,
$D_t$ are $\mathcal F_t$-measurable, while $\E[n_{t,i}\mid\mathcal F_t]=0$. By
the tower property,
\[
\E\Big\langle\sum_i(u_i-w_{t,i})n_{t,i},D_t\Big\rangle
=\E\Big[\Big\langle\sum_i(u_i-w_{t,i})\,\E[n_{t,i}\mid\mathcal F_t],\,D_t\Big\rangle\Big]=0 .
\]
This is exactly what evaluating the oracle at $\mathcal M_{t-1}$ rather than
$\mathcal M_t$ buys: with $D_t=\lmo(\mathcal M_t)$ the direction would depend on
$\xi_t$ and this term would not vanish.

\emph{Step 3: the mirror-descent error term.} By H\"older and
Assumption~\ref{as:gradbound},
$\|g_t\|_\infty\le R\max_i\dn{\nabla f_i(x_t;\xi_{t,i})}\le R\big(M+\zeta N_t\big)$
with $N_t:=\max_i\|n_{t,i}\|_2$, so
$\E\|g_t\|_\infty^2\le2R^2\big(M^2+\zeta^2\E N_t^2\big)\le2R^2\big(M^2+\zeta^2K\sigma^2\big)$,
using $\E N_t^2\le\sum_i\E\|n_{t,i}\|_2^2\le K\sigma^2$. Hence
\[
\frac{\eta}{2RT}\sum_t\E\|g_t\|_\infty^2\le\eta R\big(M^2+\zeta^2K\sigma^2\big)
=T^{-3/4}\Big(1+\frac{\zeta^2K\sigma^2}{M^2}\Big).
\]
This is one of the two places where $K$ appears polynomially; both are of lower
order than $T^{-1/4}$.

\emph{Step 4: drift and noise terms.} Let $B=L\gamma R+M\omega$ as in the
proof of Lemma~\ref{lem:track}. Using $\dn{\lambda_{-1}}=\dn{A_0w_0}\le M$,
$\Theta_0=0$, \eqref{eq:lambda-bound} for $t\ge1$ and $\E\dn{\Theta_t}\le B$
for $t\ge1$, together with $\sum_{t\ge0}(1-\alpha)^t\le1/\alpha$,
\[
\frac2T\sum_{t=0}^{T-1}\E\big(\dn{\lambda_{t-1}}+\dn{\Theta_t}\big)
\le\frac{2M}{T}+2\zeta\sigma\sqrt\alpha+\frac{2\zeta\sigma}{\alpha T}+\frac{2B}{\alpha}+2B .
\]
With the stated step sizes,
\begin{align*}
M\omega&=16\,T^{-3/4}\Big(1+\frac{\zeta\sigma}{M}\Big)+\frac{8K\zeta^2\sigma^2}{M^3}\,T^{-3/2},\\
B&=L\,T^{-3/4}+16\,T^{-3/4}\Big(1+\frac{\zeta\sigma}{M}\Big)+\frac{8K\zeta^2\sigma^2}{M^3}\,T^{-3/2},
\end{align*}
so that
$2B/\alpha=2\big(L+16+16\zeta\sigma/M\big)T^{-1/4}+16K\zeta^2\sigma^2M^{-3}T^{-1}$
and $2B\le2(L+16+16\zeta\sigma/M)T^{-3/4}+16K\zeta^2\sigma^2M^{-3}T^{-3/2}$.

\emph{Step 5: collecting.} Substituting into \eqref{eq:master-stoch} after
taking expectations,
\begin{align}
\frac1T\sum_{t=0}^{T-1}\E\dn{A_tw_t}
\le\ &\frac{\Delta+2L+32+2\zeta\sigma\,(1+16/M)}{T^{1/4}}
\;+\;\frac{2\zeta\sigma}{T^{1/2}}\notag\\
&+\frac{1+\tfrac52L+32(1+\zeta\sigma/M)+\zeta^2K\sigma^2/M^2}{T^{3/4}}
\;+\;\frac{2M+32K\zeta^2\sigma^2/M^3}{T}, \label{eq:stoch-explicit}
\end{align}
where in the last term we used $T^{-3/2}\le T^{-1}$. This is
\eqref{eq:stoch}; setting the right-hand side to $\epsilon$ yields
$T=O(\epsilon^{-4})$. The number of objectives does not enter the $T^{-1/4}$
term (the comparator is uniform); the two terms carrying $K$ polynomially are of order $T^{-3/4}$
(from the $\ell_\infty$ error term of mirror descent, Step 3) and $T^{-1}$
(from the second-order part of the weight increment, Lemma~\ref{lem:wstep-app}(b)).

\emph{Step 6: the CA distance.} The inequality
$\mathrm{CA}_t\le2\dn{A_tw_t}$ of Corollary~\ref{cor:ca} is pathwise and holds
for every measurable selection $w^\star(x_t)$, so averaging and taking
expectations gives $\frac1T\sum_t\E\,\mathrm{CA}_t\le2\times$\eqref{eq:stoch-explicit}.

Finally, note that with the deterministic choice $\eta R=\Theta(T^{-1/2}/M)$
the term $M\omega/\alpha$ would contain $16\eta RM^2/\alpha=\Theta(M)$, which
does not vanish as $T\to\infty$; this is why the weight step size must be
reduced to $\Theta(T^{-3/4}/M^2)$ in the stochastic setting.
\end{proof}

\begin{corollary}[Stochastic inner regret]\label{cor:regret-stoch}
Under the conditions of Theorem~\ref{thm:stoch}, for every $u\in\simplex$,
$\frac1T\sum_t\E\dn{\gF(x_t)w_t}\le\frac1T\sum_t\E\dn{\gF(x_t)u}$ plus the
right-hand side of \eqref{eq:stoch} with $\Delta$ replaced by $M^2\ln K$; hence
$\frac1T\sum_t\E\,\mathrm{gap}_t$ is bounded by the same quantity plus
$\bar V_T:=\min_{u\in\simplex}\E\frac1T\sum_t\big[\dn{\gF(x_t)u}-\min_w\dn{\gF(x_t)w}\big]$.
\end{corollary}

Note $\bar V_T\ge\E V_T$: the comparator must be deterministic for the noise
cross term to vanish, so the stochastic statement is against the best fixed
weighting \emph{in expectation}, not against the realized hindsight
minimizer.

The linear-$K$ terms in $C_3,C_4$ are two maxima over tasks of noise
magnitudes bounded with second moments only, and are intrinsic there (noise of
size $\sqrt K\sigma$ with probability $1/K$ makes some task an outlier at every
step). A tail assumption removes them.

\begin{corollary}[Polylogarithmic $K$ under sub-Gaussian noise]\label{cor:subg}
Replace the variance bound in Assumption~\ref{as:noise} by the conditional tail
bound $\Pr\big(\|\nabla f_i(x_t;\xi_t)-\nabla f_i(x_t)\|_2>s\mid\mathcal F_t\big)\le2e^{-s^2/(2\sigma^2)}$
for all $s\ge0$ and $i$. Then Theorem~\ref{thm:stoch} holds for
$T^{3/2}\ge(8\zeta^2\sigma^2/M^4)\ln(4KT)$ with $\sigma\to2\sigma$ in $C_1,C_3$,
$C_2=4\zeta\sigma+2M$, $\zeta^2K\sigma^2/M^2\to2\zeta^2\sigma^2(1+\ln2K)/M^2$ in
$C_3$, and $C_4=4M$ (Appendix~\ref{app:stoch}): under these conditions every
displayed constant depends on $K$ only through $\ln K$, so $T=O(\epsilon^{-4})$
with a $K$-free constant for $K\le\exp(O(\epsilon^{-2}))$;
Corollary~\ref{cor:regret-stoch} holds with the same constants plus $M^2\ln K$.
For a.s.\ bounded noise the constants of Theorem~\ref{thm:stoch} are free of $K$.
\end{corollary}

\begin{proof}[Proof of Corollary~\ref{cor:regret-stoch}]
Fix a deterministic $u\in\simplex$ and write $\psi_t(w)=\dn{A_tw}$ and
$\bar g_{t,i}:=\ip{D_t}{\nabla f_i(x_t)}$ for the \emph{noiseless} scores, so
that $g_{t,i}=\bar g_{t,i}+\ip{D_t}{n_{t,i}}$. Because $D_t=\lmo(\mathcal M_{t-1})$
is not an element of $\lmo(A_tw_t)$, Lemma~\ref{lem:subgrad} does not apply
directly; we argue from the definition of the dual norm instead. For any
$v\in\simplex$, $\pn{D_t}\le R$ gives $\psi_t(v)=\dn{A_tv}\ge-\ip{A_tv}{D_t}/R
=-\ip{\bar g_t}{v}/R$; at $v=u$ this is
$\psi_t(u)\ge-\ip{\bar g_t}{u}/R$. For $w_t$, the first part of the proof of
Lemma~\ref{lem:onestep-stoch} gives
$\ip{A_tw_t}{D_t}\le-R\dn{A_tw_t}+2R(\dn{\lambda_{t-1}}+\dn{\Theta_t})$, i.e.
$\psi_t(w_t)\le-\ip{\bar g_t}{w_t}/R+2(\dn{\lambda_{t-1}}+\dn{\Theta_t})$.
Subtracting,
\begin{align*}
\psi_t(w_t)-\psi_t(u)&\le\frac1R\ip{\bar g_t}{u-w_t}+2\big(\dn{\lambda_{t-1}}+\dn{\Theta_t}\big)\\
&=\frac1R\ip{g_t}{u-w_t}-\frac1R\Big\langle\sum_i(u_i-w_{t,i})n_{t,i},D_t\Big\rangle
+2\big(\dn{\lambda_{t-1}}+\dn{\Theta_t}\big).
\end{align*}
Sum over $t$, divide by $T$, bound $\frac1{RT}\sum_t\ip{g_t}{u-w_t}$ by
$\frac{\ln K}{\eta RT}+\frac{\eta}{2RT}\sum_t\|g_t\|_\infty^2$ via
\eqref{eq:md}, and take expectations. The signed noise term vanishes exactly as
in Step~2 of the proof of Theorem~\ref{thm:stoch}; the remaining terms are
precisely those of \eqref{eq:master-stoch} without $\Delta/(\gamma RT)$ and
$L\gamma R/2$ but with $\KL(u\|w_0)/(\eta RT)\le\ln K/(\eta RT)=M^2\ln K\,T^{-1/4}$
for the general comparator, and Steps~3--5 of that proof bound the rest by
\eqref{eq:stoch-explicit} with $\Delta$ removed and $2L$ (from $2B/\alpha$)
retained. For the gap statement, fix a deterministic $u\in\simplex$ and write,
as in the proof of Corollary~\ref{cor:gap},
$\E\,\mathrm{gap}_t=\E[\psi_t(w_t)-\psi_t(u)]+\E[\psi_t(u)-\min_w\psi_t(w)]$;
average over $t$, bound the first average by the regret bound just proved,
and then take the minimum over deterministic $u$ of the second average, which
is $\bar V_T$. The comparator must be deterministic because Step~2 of the proof
of Theorem~\ref{thm:stoch} uses $\E[n_{t,i}\mid\mathcal F_t]=0$ against a
$\mathcal F_t$-measurable $u-w_t$; the realized hindsight minimizer
$u(\omega)=\argmin_u\frac1T\sum_t\psi_t(u)$ depends on the whole trajectory and
is not admissible. Since $\E\min_u(\cdot)\le\min_u\E(\cdot)$, one has
$\E V_T\le\bar V_T$, and the bound with $\E V_T$ in place of $\bar V_T$ is
\emph{not} established by this argument.
\end{proof}

\begin{proof}[Proof of Corollary~\ref{cor:subg}]
Write $X_i:=\|n_{t,i}\|_2$, so $\Pr(X_i>s\mid\mathcal F_t)\le2e^{-s^2/(2\sigma^2)}$.
Then $\E[X_i^2\mid\mathcal F_t]=\int_0^\infty\Pr(X_i^2>u)\,du\le\int_0^\infty2e^{-u/(2\sigma^2)}du=4\sigma^2$,
so Assumption~\ref{as:noise} holds with $\sigma$ replaced by $2\sigma$, and
every step of the proofs of Lemma~\ref{lem:wstep-app}, Lemma~\ref{lem:track}
and Theorem~\ref{thm:stoch} goes through with that substitution except the two
that bound a maximum over tasks. Those improve as follows.

\emph{The $\ell_\infty$ error term (Step~3).} With $N_t=\max_iX_i$ and
$u_0:=2\sigma^2\ln(2K)$,
\[
\E[N_t^2\mid\mathcal F_t]=\int_0^\infty\Pr(N_t^2>u)\,du
\le u_0+\int_{u_0}^\infty2K e^{-u/(2\sigma^2)}du
=u_0+4K\sigma^2e^{-u_0/(2\sigma^2)}=2\sigma^2\big(1+\ln2K\big),
\]
by the union bound. Hence
$\E\|g_t\|_\infty^2\le2R^2\big(M^2+2\zeta^2\sigma^2(1+\ln2K)\big)$ and the
Step~3 term is $T^{-3/4}\big(1+2\zeta^2\sigma^2(1+\ln2K)/M^2\big)$.

\emph{The union bound in Lemma~\ref{lem:wstep-app}(b).} Since
$|\nu_j|\le\eta R\zeta X_j$,
\[
\Pr(\mathcal A^c\mid\mathcal F)\le\sum_j\Pr\Big(X_j>\frac{1}{2\eta R\zeta}\Big)
\le2K\exp\Big(-\frac{1}{8\eta^2R^2\zeta^2\sigma^2}\Big),
\]
so $\omega=16\eta R(M+2\zeta\sigma)+4K\exp\big(-1/(8\eta^2R^2\zeta^2\sigma^2)\big)$,
using $\E[X_i\mid\mathcal F]\le(\E[X_i^2\mid\mathcal F])^{1/2}\le2\sigma$ in
the first term. With $\eta R=T^{-3/4}/M^2$ the exponent is
$-M^4T^{3/2}/(8\zeta^2\sigma^2)$, and the hypothesis
$T^{3/2}\ge(8\zeta^2\sigma^2/M^4)\ln(4KT)$ is exactly
$4K\exp(-M^4T^{3/2}/(8\zeta^2\sigma^2))\le1/T$. Hence in the proof of
Theorem~\ref{thm:stoch}, $M\omega\le16T^{-3/4}(1+2\zeta\sigma/M)+M/T$, so
$2B/\alpha\le2(L+16+32\zeta\sigma/M)T^{-1/4}+2MT^{-1/2}$ and
$2B\le2(L+16+32\zeta\sigma/M)T^{-3/4}+2M/T$.

\emph{Collecting.} Repeating Step~5 with $\sigma\to2\sigma$ in the noise terms
($2\zeta\sigma\sqrt\alpha\to4\zeta\sigma T^{-1/4}$, $2\zeta\sigma/(\alpha T)\to4\zeta\sigma T^{-1/2}$)
and the two replacements above gives the constants in the statement:
$C_1=\Delta+2L+32+4\zeta\sigma(1+16/M)$, $C_2=4\zeta\sigma+2M$,
$C_3=1+\tfrac52L+32(1+2\zeta\sigma/M)+2\zeta^2\sigma^2(1+\ln2K)/M^2$,
$C_4=2M+2M=4M$. The requirement $C_3\le16C_1^3/\epsilon^2$ for a $K$-free
iteration count now reads $\ln(2K)\lesssim\epsilon^{-2}$. For almost surely
bounded noise $X_i\le\hat\sigma$ one has $N_t\le\hat\sigma$ and
$\Pr(\mathcal A^c)=0$ once $\eta R\zeta\hat\sigma\le\tfrac12$, so no maximum over
tasks is ever taken and the constants are those of Theorem~\ref{thm:stoch} with
$\sigma\to\hat\sigma$ and the $K$-terms deleted. The regret corollary adds
$\KL(u\|w_0)/(\eta RT)\le M^2\ln K\,T^{-1/4}$ as before.
\end{proof}

\begin{remark}[On the gradient-bound assumption]\label{rem:gradbound}
Assumption~\ref{as:gradbound} is used in three places: to bound the
$w$-component of the drift ($\dn{A_s(w_{s-1}-w_s)}\le M\|w_s-w_{s-1}\|_1$), in
Lemma~\ref{lem:wstep-app}(b), and in Step~3 above. In the deterministic setting the
analogous bound is \emph{derived} in Lemma~\ref{lem:M} by an induction on the
function values along the trajectory. That induction does not transfer to
Algorithm~\ref{alg:stoch}: the stochastic one-step inequality contains the
random terms $\dn{\lambda_{t-1}}$ and the noise cross term, which are
controlled in expectation but not almost surely, so it yields a bound on
$\E f_i(x_t)$ and not the almost-sure gradient bound that the drift argument
requires. Removing the assumption would need a stopping-time argument in the
manner of \citet{zhang2025mgda}, which we leave aside. Note that under an
almost-sure bound on the \emph{stochastic} gradients,
$\dn{\nabla f_i(x;\xi)}\le\widehat M$ for all $x,\xi$,
Assumption~\ref{as:gradbound} holds with $M=\widehat M$ directly, since
$\dn{\nabla f_i(x)}=\dn{\E_\xi\nabla f_i(x;\xi)}\le\E_\xi\dn{\nabla f_i(x;\xi)}$
by Jensen; in that case the assumption is not additional.
\end{remark}

\section{Per-task tracking variant}\label{app:pertask}

For completeness we record the alternative referred to in
Section~\ref{sec:stoch}; this is an upper-bound comparison of the two tracking
schemes, not a convergence theorem. Track $m_{t,i}=(1-\alpha)m_{t-1,i}+\alpha\nabla
f_i(x_t;\xi_{t,i})$ per task and set $\widehat G_t=\sum_iw_{t,i}m_{t,i}$,
$D_t=\lmo(\widehat G_t)$, $g_{t,i}=\ip{D_t}{m_{t,i}}$. Writing
$\lambda_{t,i}=m_{t,i}-\nabla f_i(x_t)$ and
$\Lambda_t=[\lambda_{t,1},\dots,\lambda_{t,K}]$, the recursion for each
$\lambda_{t,i}$ has drift $\nabla f_i(x_{t-1})-\nabla f_i(x_t)$ only, bounded by
$L\gamma R$ deterministically, so no $w$-drift term appears and
$\eta R=\Theta(T^{-1/2}/M)$ remains admissible, which gives the better constant
$M\sqrt{\log K}$ in place of $M^2\log K$. The one-step inequality however
involves $\dn{\Lambda_tw_t}=\dn{\sum_iw_{t,i}\lambda_{t,i}}$, and $w_t$ is
correlated with the past noise that produced $\lambda_{t,i}$: the weight update
raises $w_{t,i}$ precisely when past scores $\ip{D_s}{n_{s,i}}$ were positive,
and those same noise vectors appear in $\lambda_{t,i}$. The Jensen argument of
\eqref{eq:nfree} bounds $\E\|\sum_iw_{t,i}\lambda_{t,i}\|^2$ by
$\sum_i\E[w_{t,i}\|\lambda_{t,i}\|^2]$, which is not $\le\max_i\E\|\lambda_{t,i}\|^2$
when $w_t$ and $\lambda_{t,i}$ are dependent. One is left with the pathwise
bound $\dn{\Lambda_tw_t}\le\max_i\dn{\lambda_{t,i}}$ and a maximal inequality
\[
\E\max_i\|\nu_{t,i}\|_2\le\Big(\sum_i\E\|\nu_{t,i}\|_2^2\Big)^{1/2}
\le\sqrt K\max_i\big(\E\|\nu_{t,i}\|_2^2\big)^{1/2},
\]
so a factor $c_K=\sqrt K$ under Assumption~\ref{as:noise} alone (improving to
$O(\sqrt{\log K})$ if each $n_{t,i}$ is sub-Gaussian) multiplies $\zeta\sigma$
in the leading term. The two variants therefore trade a better $M$-dependence
against a worse $K$-dependence, and only the blended form has a leading term
free of $K$ under bounded variance alone. Whether $c_K$ can be removed for
per-task tracking is open.

\section{Iteration-wise inner gap: proofs for Section~\ref{sec:iterwise}}\label{app:reg}

Throughout, $\pn\cdot$ is an arbitrary norm, $A_t:=\gF(x_t)$, and
\[
f_t(w):=\dn{A_tw},\qquad r(w):=\KL(w\|u_K),\qquad
\tilde\psi_t:=f_t+\tau r,
\]
\[
p_t:=\argmin_{w\in\simplex^\kappa}\tilde\psi_t(w),\qquad \Phi_t:=\KL(p_t\|w_t),
\]
with $\simplex^\kappa=\{w\in\simplex:w_i\ge\kappa\}$, $\Lambda:=\ln(1/\kappa)$,
$\bar\eta:=\eta R$, $\bar\gamma:=\gamma R$, $\rho:=\bar\eta\tau$, and
$\mathrm{gap}_t:=f_t(w_t)-\min_{\simplex}f_t$. For $\tau>0$, $\tilde\psi_t$
is strictly convex (as $r$ is) and continuous on the compact set
$\simplex^\kappa$, so $p_t$ exists and is unique. We use three facts about
$f_t$, valid in every norm along a trajectory with
$\max_i\dn{\nabla f_i(x_t)}\le M$:
\begin{enumerate}[leftmargin=1.8em]
\item[(F1)] $f_t$ is convex and $M$-Lipschitz in $\|\cdot\|_1$ on $\simplex$:
$|f_t(v)-f_t(w)|\le\dn{A_t(v-w)}\le\sum_i|v_i-w_i|\dn{\nabla f_i(x_t)}\le M\|v-w\|_1$.
\item[(F2)] $\hat g_t:=-g_t/R\in\partial f_t(w_t)$ (Lemma~\ref{lem:subgrad}), and
$\|\hat g_t\|_\infty\le\max_i\dn{\nabla f_i(x_t)}\le M$ by H\"older.
\item[(F3)] For any $w\in\simplex$ and $x,x'$,
$|\dn{\gF(x)w}-\dn{\gF(x')w}|\le\dn{(\gF(x)-\gF(x'))w}\le\sum_iw_i\dn{\nabla f_i(x)-\nabla f_i(x')}\le L\pn{x-x'}$
(Assumption~\ref{as:smooth}).
\end{enumerate}

\begin{lemma}[Closed form of line~6]\label{lem:closedform}
The minimizer in line~6 of Algorithm~\ref{alg:main} is given by
\eqref{eq:regeg}: $w_{t+1,i}=\max\{\kappa,c_i/Z_t\}$ with
$c_i=w_{t,i}^{1/(1+\rho)}e^{\eta g_{t,i}/(1+\rho)}$ and $Z_t>0$ the unique
normalizer for which $\sum_iw_{t+1,i}=1$ (for $\kappa<1/K$; if $\kappa=1/K$
then $w_{t+1}=u_K$).
\end{lemma}

\begin{proof}
The objective $h(w)=-\eta\ip{g_t}{w}+\rho r(w)+\KL(w\|w_t)$ is strictly convex
and differentiable on the relative interior, and the constraints are affine,
so the KKT conditions are necessary and sufficient. With multiplier $\nu$ for
$\sum_iw_i=1$ and $\lambda_i\ge0$ for $w_i\ge\kappa$, and using
$\partial_i r(w)=\ln w_i+\ln K+1$, $\partial_i\KL(w\|w_t)=\ln w_i-\ln w_{t,i}+1$,
stationarity reads
$-\eta g_{t,i}+(1+\rho)\ln w_i+\rho(\ln K+1)+1-\ln w_{t,i}-\nu-\lambda_i=0$,
i.e. $(1+\rho)\ln w_i=\ln w_{t,i}+\eta g_{t,i}+\text{const}+\lambda_i$. If
$\lambda_i=0$ this gives $w_i=c_i/Z_t$ for a constant $Z_t$; if $\lambda_i>0$
then $w_i=\kappa$, and $\lambda_i\ge0$ is equivalent to $c_i/Z_t\le\kappa$.
Hence $w_{t+1,i}=\max\{\kappa,c_i/Z_t\}$. The map $z\mapsto\sum_i\max\{\kappa,c_i/z\}$
is continuous, non-increasing, tends to $K\kappa<1$ as $z\to\infty$ and to
$\infty$ as $z\to0$, and is strictly decreasing wherever its value exceeds
$K\kappa$, so $Z_t$ is unique.
\end{proof}

\subsection{Computing the clipped update}\label{app:clip}

Equation~\eqref{eq:regeg} determines $w_{t+1}$ implicitly through the
normalizer $Z_t$, or equivalently through the set $S^\star:=\{i:c_i/Z_t<\kappa\}$
of clipped coordinates. Given $S^\star$, $Z_t$ is explicit: the unclipped
coordinates must carry the mass $1-|S^\star|\kappa$, so
$Z_t=\sum_{i\notin S^\star}c_i/(1-|S^\star|\kappa)$. The clipped set is found by
the following recurrence, which is the ``clip, renormalize, repeat'' rule
stated as a fixed-point iteration. Write $c=(c_i)_{i\in[K]}$ with
$c_i=w_{t,i}^{1/(1+\rho)}e^{\eta g_{t,i}/(1+\rho)}>0$ and, for $S\subseteq[K]$
with $|S|\kappa<1$,
\begin{equation}\label{eq:gS}
Z(S):=\frac{\sum_{i\notin S}c_i}{1-|S|\kappa}.
\end{equation}

\begin{algorithm}[H]
\caption{Clipped renormalization (KL projection of the softmax output onto $\simplex^\kappa$)}
\label{alg:clip}
\begin{algorithmic}[1]
\Require $c\in\R^K_{>0}$; $\kappa\in[0,1/K)$
\State $S_0\gets\emptyset$
\For{$k=0,1,2,\dots$}
  \State $Z_k\gets Z(S_k)$ \Comment{renormalize the unclipped coordinates, \eqref{eq:gS}}
  \State $A_k\gets\{i\notin S_k:\ c_i<\kappa Z_k\}$ \Comment{coordinates that fall below the floor}
  \If{$A_k=\emptyset$} \Return $w_i=\kappa$ for $i\in S_k$,\quad $w_i=c_i/Z_k$ for $i\notin S_k$ \EndIf
  \State $S_{k+1}\gets S_k\cup A_k$
\EndFor
\end{algorithmic}
\end{algorithm}

\begin{lemma}[Correctness and termination of Algorithm~\ref{alg:clip}]\label{lem:clip}
Let $\kappa\in[0,1/K)$ and $c\in\R^K_{>0}$, and let $Z^\star$ be the unique
normalizer of Lemma~\ref{lem:closedform}, with
$S^\star=\{i:c_i/Z^\star<\kappa\}$. Along Algorithm~\ref{alg:clip}:
\begin{enumerate}[leftmargin=1.8em]
\item[(i)] $S_0\subseteq S_1\subseteq\cdots\subseteq S^\star$, and $Z_0<Z_1<\cdots\le Z^\star$
while the recurrence continues;
\item[(ii)] the recurrence returns after at most $|S^\star|+1\le K$ rounds, in
the first round $k$ with $A_k=\emptyset$; at that round $S_k=S^\star$ and
$Z_k=Z^\star$, so the returned $w$ is $w_{t+1}$ of \eqref{eq:regeg}. (A
coordinate with $c_i=\kappa Z^\star$ exactly may sit in either set; the
returned vector is the same.)
\end{enumerate}
Each round costs $O(K)$, so the update costs $O(K\min\{K,|S^\star|+1\})$; sorting
$c$ once gives an $O(K\log K)$ implementation, since $S^\star$ is a lower set in
the order of $c$. At $\kappa=0$ the recurrence stops after one round with
$S_0=\emptyset$ and $Z_0=\sum_ic_i$, i.e. the plain softmax.
\end{lemma}

\begin{proof}
Two facts about $Z(\cdot)$ are used. First, by the definition of $S^\star$ and
$\sum_i\max\{\kappa,c_i/Z^\star\}=1$,
\begin{equation}\label{eq:Zstar}
\sum_{i\notin S^\star}c_i=Z^\star\big(1-|S^\star|\kappa\big),\qquad
c_i<\kappa Z^\star\ \ (i\in S^\star),\qquad
c_i\ge\kappa Z^\star\ \ (i\notin S^\star).
\end{equation}
Second, for $S\subseteq S^\star$,
\[
\sum_{i\notin S}c_i=\sum_{i\notin S^\star}c_i+\sum_{i\in S^\star\setminus S}c_i
\le Z^\star(1-|S^\star|\kappa)+\kappa Z^\star\,(|S^\star|-|S|)=Z^\star(1-|S|\kappa),
\]
with strict inequality if $S\ne S^\star$; dividing by $1-|S|\kappa>0$ gives
\begin{equation}\label{eq:ZS}
Z(S)\le Z^\star\quad\text{for all }S\subseteq S^\star,\qquad\text{with equality iff }S=S^\star .
\end{equation}

\emph{(i) Monotonicity.} We show by induction that $S_k\subseteq S^\star$. This
holds for $S_0=\emptyset$. If $S_k\subseteq S^\star$ then $Z_k=Z(S_k)\le Z^\star$
by \eqref{eq:ZS}, so every $i\in A_k$ satisfies $c_i<\kappa Z_k\le\kappa Z^\star$,
i.e. $i\in S^\star$; hence $S_{k+1}\subseteq S^\star$. The sets are nested by
construction. For the normalizers, if $A_k\ne\emptyset$ then, using
$c_i<\kappa Z_k$ for $i\in A_k$,
\[
\sum_{i\notin S_{k+1}}c_i=Z_k(1-|S_k|\kappa)-\sum_{i\in A_k}c_i
>Z_k(1-|S_k|\kappa)-|A_k|\kappa Z_k=Z_k\big(1-|S_{k+1}|\kappa\big),
\]
and $1-|S_{k+1}|\kappa\ge1-|S^\star|\kappa>0$, so $Z_{k+1}>Z_k$.

\emph{(ii) Termination and correctness.} Since the $S_k$ are nested subsets of
$S^\star$ and strictly grow while $A_k\ne\emptyset$, the recurrence stops after
at most $|S^\star|+1$ rounds. At termination, $A_k=\emptyset$ means
$c_i\ge\kappa Z_k$ for all $i\notin S_k$, while every $i\in S_k$ was added at
some round $j<k$ with $c_i<\kappa Z_j\le\kappa Z_k$ by (i). Hence the returned
vector satisfies $w_i=\max\{\kappa,c_i/Z_k\}$ for every $i$, and it sums to $1$
by the definition of $Z(S_k)$. It is therefore the water-filling vector of
Lemma~\ref{lem:closedform} with normalizer $Z_k$; by the uniqueness of the
normalizer proved there, $Z_k=Z^\star$, and consequently $S_k=S^\star$ and
$w=w_{t+1}$.

For the complexity claim, note that $S^\star$ consists of the coordinates with
the smallest $c_i$: if $c_i\le c_j$ and $j\in S^\star$ then
$c_i/Z^\star\le c_j/Z^\star<\kappa$. After sorting, one can scan the candidates
$S=\{$the $s$ smallest$\}$ for $s=0,1,\dots$ and stop at the first $s$ for which
the $(s+1)$-st smallest $c$ is at least $\kappa Z(S)$; by \eqref{eq:ZS} and the
argument above this is $S^\star$.
\end{proof}

At $\tau=\kappa=0$ the update is a single softmax, and for $\kappa>0$ the
recurrence typically stops after one or two rounds, so the per-iteration cost of
line~6 is $O(K)$ in practice and never more than $O(K^2)$; it is negligible
against the $K$ backward passes.

\begin{remark}[Why the unsquared objective]\label{rem:unsquared}
In Euclidean geometry, the squared inner objective
$\tfrac12\|\gF(x_t)w\|_2^2$ is smooth. The analyses of
\citet{zhang2025mgda} exploit this structure for iteration-wise CA distance.
Our oracle instead supplies a subgradient of the unsquared norm, matching the
linear descent term in Lemma~\ref{lem:onestep}. In Euclidean geometry, the
squared-objective gradient is $\|G_t\|_2$ times the unsquared gradient when
$G_t\ne0$; the latter involves normalization by $\|G_t\|_2$ and is
ill-conditioned near a zero blend.

For general dual norms, nonsmoothness can also occur at nonzero blends: the
$\ell_1$ norm is nonsmooth at vectors with zero coordinates, and the nuclear
norm at rank-deficient matrices. Squaring need not remove this nonsmoothness.
Our analysis therefore uses only convexity, Lipschitz continuity, and the
oracle subgradient identity. It incurs the slack $2\bar\eta^2M^2$ per step in
Lemma~\ref{lem:contract-slack}, a H\"older-$1/2$ sensitivity bound, and the
$1/\kappa$ factor in Lemma~\ref{lem:sens}. These bounds lead to the stated
burn-in and total-horizon exponents; sharper analysis under additional norm
smoothness assumptions remains open.
\end{remark}

\subsection{Contraction with slack}

\begin{lemma}[Contraction toward the regularized target]\label{lem:contract-slack}
Suppose $\max_i\dn{\nabla f_i(x_t)}\le M$. For every $u\in\simplex^\kappa$,
\[
(1+\rho)\KL(u\|w_{t+1})\;\le\;\KL(u\|w_t)+\bar\eta\big[\tilde\psi_t(u)-\tilde\psi_t(w_{t+1})\big]+2\bar\eta^2M^2 .
\]
In particular, since $p_t$ minimizes $\tilde\psi_t$ over $\simplex^\kappa\ni w_{t+1}$,
\[
\KL(p_t\|w_{t+1})\;\le\;\frac{\Phi_t+2\bar\eta^2M^2}{1+\rho}.
\]
\end{lemma}

\begin{proof}
Start from \eqref{eq:md1} in the proof of Lemma~\ref{lem:md}, written with
$\hat g_t=-g_t/R$, so that $\eta\ip{g_t}{u-w_{t+1}}=-\bar\eta\ip{\hat g_t}{u-w_{t+1}}$:
\[
(1+\rho)\KL(u\|w_{t+1})\le\KL(u\|w_t)-\KL(w_{t+1}\|w_t)
+\bar\eta\ip{\hat g_t}{u-w_{t+1}}+\rho\big(r(u)-r(w_{t+1})\big).
\]
For the inner product, split $u-w_{t+1}=(u-w_t)+(w_t-w_{t+1})$ and use (F2):
the subgradient inequality gives $\ip{\hat g_t}{u-w_t}\le f_t(u)-f_t(w_t)$, and
H\"older gives $\ip{\hat g_t}{w_t-w_{t+1}}\le M\|w_{t+1}-w_t\|_1$. Then, by (F1),
$f_t(u)-f_t(w_t)=f_t(u)-f_t(w_{t+1})+\big(f_t(w_{t+1})-f_t(w_t)\big)\le f_t(u)-f_t(w_{t+1})+M\|w_{t+1}-w_t\|_1$.
Hence
\begin{align*}
\bar\eta\ip{\hat g_t}{u-w_{t+1}}&\le\bar\eta\big(f_t(u)-f_t(w_{t+1})\big)+2\bar\eta M\|w_{t+1}-w_t\|_1\\
&\le\bar\eta\big(f_t(u)-f_t(w_{t+1})\big)+2\bar\eta^2M^2+\tfrac12\|w_{t+1}-w_t\|_1^2 ,
\end{align*}
by Young's inequality $2ab\le2a^2+\tfrac12b^2$ with $a=\bar\eta M$,
$b=\|w_{t+1}-w_t\|_1$. Pinsker's inequality bounds the last term by
$\KL(w_{t+1}\|w_t)$, which cancels the $-\KL(w_{t+1}\|w_t)$ above.
Substituting and using $\rho=\bar\eta\tau$,
\[
(1+\rho)\KL(u\|w_{t+1})\le\KL(u\|w_t)+\bar\eta\big[f_t(u)+\tau r(u)-f_t(w_{t+1})-\tau r(w_{t+1})\big]+2\bar\eta^2M^2 ,
\]
which is the claim. The second statement follows from
$\tilde\psi_t(p_t)\le\tilde\psi_t(w_{t+1})$, as $w_{t+1}\in\simplex^\kappa$ by
construction.
\end{proof}

The additive slack $2\bar\eta^2M^2$ arises from our subgradient analysis of
the potentially nonsmooth objective $f_t$. This argument does not establish
contraction without slack. Euclidean squared-objective analyses exploit
additional smoothness; such smoothness is not available for every dual norm.

\subsection{The moving target}

\begin{lemma}[Sensitivity of the regularized optimum]\label{lem:sens}
Let $p(x):=\argmin_{\simplex^\kappa}\{\dn{\gF(x)w}+\tau r(w)\}$. For any $x,x'$,
\[
\|p(x)-p(x')\|_1\;\le\;\sqrt{\frac{2L\pn{x-x'}}{\tau}} ,
\]
and hence, along the trajectory,
\[
\chi^2(p_t\|p_{t+1}):=\sum_i\frac{(p_{t,i}-p_{t+1,i})^2}{p_{t+1,i}}\;\le\;\frac{2L\bar\gamma}{\tau\kappa}.
\]
\end{lemma}

\begin{proof}
Write $p=p(x)$, $p'=p(x')$, $\tilde\psi_x=f_x+\tau r$ with $f_x=\dn{\gF(x)\cdot}$.
Since $f_x$ is convex and finite, the first-order optimality condition for the
convex problem $\min_{\simplex^\kappa}\tilde\psi_x$ provides
$s\in\partial f_x(p)$ with $\ip{s+\tau\nabla r(p)}{p'-p}\ge0$. Convexity of
$f_x$ and the Bregman identity $r(p')=r(p)+\ip{\nabla r(p)}{p'-p}+\KL(p'\|p)$
give
\[
\tilde\psi_x(p')-\tilde\psi_x(p)\ge\ip{s+\tau\nabla r(p)}{p'-p}+\tau\KL(p'\|p)\ge\tau\KL(p'\|p)\ge\tfrac\tau2\|p'-p\|_1^2 ,
\]
the last step by Pinsker. Symmetrically
$\tilde\psi_{x'}(p)-\tilde\psi_{x'}(p')\ge\tfrac\tau2\|p-p'\|_1^2$. Adding and
cancelling the identical $\tau r$ terms,
\[
\tau\|p-p'\|_1^2\le\big[f_x(p')-f_{x'}(p')\big]-\big[f_x(p)-f_{x'}(p)\big]\le2\sup_{w\in\simplex}|f_x(w)-f_{x'}(w)|\le2L\pn{x-x'}
\]
by (F3), which gives the first claim. For the second, $\pn{x_{t+1}-x_t}\le\bar\gamma$,
and since $p_{t+1,i}\ge\kappa$,
$\chi^2(p_t\|p_{t+1})\le\kappa^{-1}\sum_i(p_{t,i}-p_{t+1,i})^2\le\kappa^{-1}\|p_t-p_{t+1}\|_1^2$.
\end{proof}

This H\"older-$\tfrac12$ bound uses only strong convexity and the uniform
closeness (F3) of $f_x$ and $f_{x'}$, so it holds in every norm. A route to
stronger sensitivity estimates would require additional regularity. Even in
Euclidean geometry, differentiating the unsquared objective introduces
normalization by the blend norm near stationarity; general norms may also
have nonzero nonsmooth points.

\begin{lemma}[KL perturbation of the first argument]\label{lem:klpert}
Let $p,p',q\in\simplex^\kappa$ and $\Lambda:=\ln(1/\kappa)$. Then
\[
\KL(p'\|q)\le\KL(p\|q)+\sqrt{(2+\Lambda)\KL(p'\|q)}\cdot\sqrt{\chi^2(p\|p')} .
\]
\end{lemma}

\begin{proof}
Fix $q$ and let $F(\cdot):=\KL(\cdot\|q)$, which is convex with
$\nabla F(p')_i=\ln(p_i'/q_i)+1$. Convexity at $p'$ gives
$F(p')-F(p)\le\ip{\nabla F(p')}{p'-p}=\sum_i\ln(p_i'/q_i)(p_i'-p_i)$, the
constant dropping because $\sum_i(p_i'-p_i)=0$. Cauchy--Schwarz with the
splitting $\sqrt{p_i'}\cdot1/\sqrt{p_i'}$ yields
\[
\sum_i\ln\frac{p_i'}{q_i}(p_i'-p_i)
\le\Big(\sum_ip_i'\ln^2\frac{p_i'}{q_i}\Big)^{1/2}
\Big(\sum_i\frac{(p_i'-p_i)^2}{p_i'}\Big)^{1/2}
=\sqrt{V(p'\|q)}\sqrt{\chi^2(p\|p')} .
\]
It remains to show $V(p'\|q)\le(2+\Lambda)\KL(p'\|q)$. Put $r_i:=p_i'/q_i$;
since $p',q\in\simplex^\kappa$ we have $r_i\in[\kappa,1/\kappa]$, so
$|\ln r_i|\le\Lambda$. Writing $\psi(r):=r\ln^2r$ and
$d(r):=r\ln r-r+1\ge0$, we have $V(p'\|q)=\sum_iq_i\psi(r_i)$ and, because
$\sum_iq_i(r_i-1)=0$, also $\KL(p'\|q)=\sum_iq_id(r_i)$. So it suffices that
$\psi(r)\le(2+\Lambda)d(r)$ whenever $|\ln r|\le\Lambda$. Substituting
$r=e^s$, this is
\begin{equation}\label{eq:scalar}
e^ss^2\le(2+\Lambda)\,(se^s-e^s+1)\qquad\text{for }|s|\le\Lambda ,
\end{equation}
where $se^s-e^s+1\ge0$ with equality only at $s=0$ (both sides of
\eqref{eq:scalar} vanish there). We prove the stronger statement
$e^ss^2\le(2+\max(s,0))(se^s-e^s+1)$ for all real $s$, which implies
\eqref{eq:scalar} since $\max(s,0)\le\Lambda$ on $[-\Lambda,\Lambda]$.

\emph{Case $s\ge0$.} Let $h(s):=(2+s)(se^s-e^s+1)-e^ss^2$. Dividing by $e^s>0$,
$e^{-s}h(s)=(2+s)(s-1+e^{-s})-s^2=s-2+(2+s)e^{-s}=:k(s)$.
Then $k(0)=0$ and $k'(s)=1-(1+s)e^{-s}\ge0$ because $1+s\le e^s$. Hence
$k\ge0$ on $[0,\infty)$, i.e. $h\ge0$.

\emph{Case $s\le0$.} Put $u:=-s\ge0$. The claim $e^ss^2\le2(se^s-e^s+1)$ reads
$e^{-u}u^2\le2\big(1-(1+u)e^{-u}\big)$, i.e. $u^2+2u+2\le2e^u$, which holds
because $2e^u=2+2u+u^2+\sum_{j\ge3}2u^j/j!\ge u^2+2u+2$.

Hence $e^ss^2\le(2+\max(s,0))(se^s-e^s+1)$ for all $s$, which gives
\eqref{eq:scalar} and therefore $V(p'\|q)\le(2+\Lambda)\KL(p'\|q)$.
\end{proof}

\subsection{The recursion, the bias, and the proof of Theorem~\ref{thm:iterwise}}

\begin{lemma}[Tracking recursion]\label{lem:recursion}
Assume $\rho=\bar\eta\tau\le1$ and $\max_i\dn{\nabla f_i(x_t)}\le M$. Then
\[
\Phi_{t+1}\;\le\;\Big(1-\frac{\rho}{4}\Big)\Phi_t+4\bar\eta^2M^2+\frac{2(2+\Lambda)}{\rho}\,\chi^2(p_t\|p_{t+1}).
\]
\end{lemma}

\begin{proof}
Apply Lemma~\ref{lem:klpert} with $p=p_t$, $p'=p_{t+1}$, $q=w_{t+1}$ (all in
$\simplex^\kappa$: $w_{t+1}$ by construction, $p_t,p_{t+1}$ by definition):
$\Phi_{t+1}\le\KL(p_t\|w_{t+1})+\sqrt{(2+\Lambda)\Phi_{t+1}}\sqrt{\chi^2}$.
Young's inequality $ab\le\tfrac\theta2a^2+\tfrac1{2\theta}b^2$ with
$a=\sqrt{\Phi_{t+1}}$, $b=\sqrt{(2+\Lambda)\chi^2}$, $\theta=\rho/2$ bounds the
last term by $\tfrac\rho4\Phi_{t+1}+(2+\Lambda)\chi^2/\rho$. Using
Lemma~\ref{lem:contract-slack} for $\KL(p_t\|w_{t+1})$ and rearranging,
\[
\Phi_{t+1}\le\frac{1}{1-\rho/4}\Big[\frac{\Phi_t+2\bar\eta^2M^2}{1+\rho}+\frac{(2+\Lambda)\chi^2}{\rho}\Big].
\]
For $\rho\in(0,1]$, $(1-\rho/4)(1+\rho)=1+\tfrac34\rho-\tfrac14\rho^2\ge1+\tfrac12\rho$,
so $[(1-\rho/4)(1+\rho)]^{-1}\le(1+\rho/2)^{-1}\le1-\rho/4$, and
$(1-\rho/4)^{-1}\le2$. Substituting gives the claim.
\end{proof}

\begin{lemma}[Bias of the regularized target]\label{lem:bias}
Suppose $\max_i\dn{\nabla f_i(x_t)}\le M$ and $\kappa\le1/K$. Then
$f_t(p_t)-\min_{\simplex}f_t\le2MK\kappa+\tau\ln K$.
\end{lemma}

\begin{proof}
Let $w^\star_t\in\argmin_{\simplex}f_t$ and $\hat w:=(1-K\kappa)w^\star_t+K\kappa u_K$.
Then $\hat w_i\ge\kappa$ and $\hat w\in\simplex$, so $\hat w\in\simplex^\kappa$;
$\|\hat w-w^\star_t\|_1=K\kappa\|u_K-w^\star_t\|_1\le2K\kappa$; and
$r(\hat w)\le(1-K\kappa)r(w^\star_t)\le\ln K$ by convexity of $r$, $r(u_K)=0$
and $r\le\ln K$. Optimality of $p_t$ and $r\ge0$ give
$f_t(p_t)\le\tilde\psi_t(p_t)\le\tilde\psi_t(\hat w)\le f_t(\hat w)+\tau\ln K
\le f_t(w^\star_t)+M\|\hat w-w^\star_t\|_1+\tau\ln K$, by (F1).
\end{proof}

\begin{proof}[Proof of Theorem~\ref{thm:iterwise}]
We first dispose of the trivial regime: if $\epsilon>2M$ then
$f_t(w)\le M<\epsilon/2$ for every $w\in\simplex$, so $\mathrm{gap}_t\le M<\epsilon$
and $\dn{\gF(x_t)w_t}\le M<\epsilon$ for every $t$. Assume $\epsilon\le2M$, so
also $\beta\le2M$. Then $\kappa=\beta/(4MK)\le1/(2K)<1/K$, so $\simplex^\kappa$
has non-empty relative interior; and
\[
\rho=\bar\eta\tau=\frac{\epsilon^2\beta^2}{1536\,M^4\ln^2K}\le\frac{16}{1536\ln^22}<1 ,
\]
so Lemma~\ref{lem:recursion} applies. Note $\Lambda=\ln(1/\kappa)=\ln(4MK/\beta)$.

\emph{(i) Tracking.} Iterating Lemma~\ref{lem:recursion} with the uniform
bound $\chi^2(p_t\|p_{t+1})\le2L\bar\gamma/(\tau\kappa)$ of Lemma~\ref{lem:sens},
from $\Phi_0=\KL(p_0\|u_K)\le\ln K$, and summing the geometric series
$\sum_{j\ge0}(1-\rho/4)^j=4/\rho$,
\begin{equation}\label{eq:Phi-bound}
\Phi_t\le\Big(1-\frac\rho4\Big)^t\ln K+\frac{16\bar\eta^2M^2}{\rho}+\frac{8(2+\Lambda)}{\rho^2}\cdot\frac{2L\bar\gamma}{\tau\kappa}
=\Big(1-\frac\rho4\Big)^t\ln K+\frac{16\bar\eta M^2}{\tau}+\frac{16(2+\Lambda)L\bar\gamma}{\bar\eta^2\tau^3\kappa}.
\end{equation}
We show each of the three terms is at most $\epsilon^2/(24M^2)$ for $t\ge T_0$.
The second equals
$16M^2\cdot\frac{\epsilon^2\beta}{768M^4\ln K}\cdot\frac{2\ln K}{\beta}=\epsilon^2/(24M^2)$
exactly. For the third, substitute
$\bar\eta^2=\epsilon^4\beta^2/(768^2M^8\ln^2K)$, $\tau^3=\beta^3/(8\ln^3K)$
and $\kappa=\beta/(4MK)$, so that
$\bar\eta^2\tau^3\kappa=\epsilon^4\beta^6/(9\cdot2^{21}M^9K\ln^5K)$ (as
$768^2\cdot8\cdot4=9\cdot2^{16}\cdot2^{5}$); then
\[
\frac{16(2+\Lambda)L\bar\gamma}{\bar\eta^2\tau^3\kappa}
=\frac{9\cdot2^{25}\,(2+\Lambda)LM^{9}K\ln^5K}{\epsilon^{4}\beta^{6}}\,\bar\gamma
=\frac{9\cdot2^{25}\,\epsilon^{2}}{27\cdot2^{28}M^{2}}=\frac{\epsilon^2}{24M^2}
\]
exactly, by the choice of $\bar\gamma$. For the first,
$(1-\rho/4)^t\le e^{-\rho t/4}\le\epsilon^2/(24M^2\ln K)$ once
$t\ge\frac4\rho\ln\frac{24M^2\ln K}{\epsilon^2}$, and
$4/\rho=6144M^4\ln^2K/(\epsilon^2\beta^2)$, so this holds for $t\ge T_0$.
Hence $\Phi_t\le\epsilon^2/(8M^2)$ for $t\ge T_0$, and Pinsker gives
$\|w_t-p_t\|_1\le\sqrt{2\Phi_t}\le\epsilon/(2M)$. Lemma~\ref{lem:bias} with
$2MK\kappa=\beta/2$ and $\tau\ln K=\beta/2$ gives $\mathrm{gap}(p_t)\le\beta$.
Finally, by (F1),
$\mathrm{gap}_t=\big(f_t(w_t)-f_t(p_t)\big)+\mathrm{gap}(p_t)\le M\|w_t-p_t\|_1+\beta\le\epsilon/2+\beta$.

\emph{(ii) Progress.} The bound $\frac1T\sum_{t<T}\dn{\gF(x_t)w_t}\le\Delta/(\bar\gamma T)+\bar\eta M^2/2+L\bar\gamma/2$
is \eqref{eq:det-master} of Theorem~\ref{thm:det}, which holds for every
$\tau\ge0$, $\kappa\le1/K$ and every $\gamma,\eta$. It remains to check the
three terms at the stated parameters. $\Delta/(\bar\gamma T)\le\epsilon/4$ for
$T\ge T_1=\lceil4\Delta/(\bar\gamma\epsilon)\rceil$. Next
$\bar\eta M^2/2=\epsilon^2\beta/(1536M^2\ln K)\le\epsilon/8$, because
$\epsilon\beta/M^2\le\epsilon^2/M^2\le4\le192\ln2\le192\ln K$. Finally
$L\bar\gamma/2\le\epsilon/8$ because, using $(2+\Lambda)K\ge1$,
$\ln K\ge\ln2$ and $\beta\le\epsilon\le2M$,
$L\bar\gamma=\epsilon^{6}\beta^{6}/(27\cdot2^{28}(2+\Lambda)M^{11}K\ln^5K)
\le\epsilon\,(\epsilon/M)^{11}/(27\cdot2^{28}\ln^52)\le\epsilon\,2^{11}/(27\cdot2^{28}\ln^52)<\epsilon/4$.
Hence the average is at most $\epsilon/4+\epsilon/8+\epsilon/8=\epsilon/2$ for
$T\ge T_1$.

\emph{(iii) Tail output.} Let $T\ge\max\{2T_0,T_1\}$. By (ii),
$\sum_{t=T_0}^{T-1}\dn{\gF(x_t)w_t}\le\sum_{t=0}^{T-1}\dn{\gF(x_t)w_t}\le T\epsilon/2$,
the first inequality by non-negativity of the dropped terms. The minimum over
the tail is at most its average, so
$\dn{\gF(x_{\hat t})w_{\hat t}}\le(T\epsilon/2)/(T-T_0)\le\epsilon$, using
$T-T_0\ge T/2$. Since $\hat t\ge T_0$, part (i) gives
$\mathrm{gap}_{\hat t}\le\epsilon/2+\beta$, and
$\min_w\dn{\gF(x_{\hat t})w}\le\dn{\gF(x_{\hat t})w_{\hat t}}\le\epsilon$. With $\beta=\epsilon/2$
the parameters reduce to $\tau=\epsilon/(4\ln K)$, $\kappa=\epsilon/(8MK)$,
$\bar\eta=\epsilon^3/(1536M^4\ln K)$,
$\bar\gamma=\epsilon^{12}/(27\cdot2^{34}(2+\Lambda)LM^{11}K\ln^5K)$ and
$T_0=\lceil(24576M^4\ln^2K/\epsilon^4)\ln(24M^2\ln K/\epsilon^2)\rceil$, which
gives the stated orders.
\end{proof}

\begin{proof}[Proof of Corollary~\ref{cor:iterwise-ca}]
By the triangle inequality,
$\mathrm{CA}_t\le\dn{\gF(x_t)(w_t-p_t)}+\dn{\gF(x_t)(p_t-w^\star(x_t))}$. The
first term is at most $M\|w_t-p_t\|_1\le\epsilon/2$ for $t\ge T_0$ by (F1)
and Theorem~\ref{thm:iterwise}(i). For the second, $w^\star(x_t)$ minimizes
$\varphi_q\circ\gF(x_t)$ over $\simplex$ (the same minimizers as $f_t$), so the
argument of Remark~\ref{rem:hilbert} applied with $w_t$ replaced by $p_t$
gives $\dn{\gF(x_t)(p_t-w^\star(x_t))}^q\le(q/\mu)M^{q-1}\mathrm{gap}(p_t)\le(q/\mu)M^{q-1}\beta$,
which equals $(\epsilon/2)^q$ for $\beta=\mu\epsilon^q/(q2^qM^{q-1})$. The
orders of $T_0$ and $T_1$ follow by substituting this $\beta$ into
Theorem~\ref{thm:iterwise}: in the Euclidean case $\beta=\epsilon^2/(8M)$,
so $\epsilon^2\beta^2=\epsilon^6/(64M^2)$ and $\epsilon^6\beta^6=\epsilon^{18}/(2^{18}M^6)$.
\end{proof}

\begin{remark}[On the gradient bound in Theorem~\ref{thm:iterwise}]\label{rem:M-reg}
Theorem~\ref{thm:iterwise} takes the trajectory gradient bound as a
hypothesis, and we keep it as one: unlike Theorem~\ref{thm:det} at
$\tau=\kappa=0$, we do not derive it. We record what a derivation would
involve. The induction of Lemma~\ref{lem:M} changes in the comparator: Lemma~\ref{lem:M} uses $u=e_i$, which is not in
$\simplex^\kappa$, so one applies \eqref{eq:md} at
$\hat e_i=(1-K\kappa)e_i+K\kappa u_K\in\simplex^\kappa$ and pays
$\gamma\|g_s\|_\infty\|e_i-\hat e_i\|_1\le2\bar\gamma MK\kappa$ per step for
the difference, together with the regularization term $\rho\ln K/\eta\cdot\gamma=\bar\gamma\tau\ln K$
per step from Lemma~\ref{lem:md}. Since $\tau\ln K+2MK\kappa=\beta\le\epsilon$
and $T_1\bar\gamma\le4\Delta/\epsilon+\bar\gamma$, over $T_1$ steps these add
at most $T_1\bar\gamma\,\beta\le4\Delta+\bar\gamma\epsilon$ to the
function-value increase, so $f_i(x_t)-f_i^\star\le5\Delta+C'$ with $C'$
collecting $\bar\gamma\epsilon$ and the (much smaller) remaining terms of
Lemma~\ref{lem:M} at these step sizes, and Lemma~\ref{lem:gradbound} converts
this into $\dn{\nabla f_i(x_t)}\le\sqrt{2L(5\Delta+C')}$, which is below the
$M$ of Lemma~\ref{lem:M} ($\ge4\sqrt{L\Delta}$) once $C'\le3\Delta$. Verifying
$C'\le3\Delta$ requires relating the remaining terms, which involve $M$,
$\epsilon$ and $\beta$ but not $\Delta$, to $\Delta$ through the explicit value
of $M$ in Lemma~\ref{lem:M} and the horizon actually run, which is
$\max\{2T_0,T_1\}$ rather than $T_1$. We have not carried this out, so the
gradient bound remains an assumption of Theorem~\ref{thm:iterwise}.
\end{remark}

\section{Uniform convexity of \texorpdfstring{$\ell_p$}{lp} and Schatten-\texorpdfstring{$p$}{p} norms}\label{app:lp}

Corollaries~\ref{cor:gap} and~\ref{cor:iterwise-ca} convert a bound on the
inner gap into a bound on the CA distance whenever the dual norm satisfies the
following uniform convexity inequality of power type $q\ge2$ with constant
$\mu\in(0,1]$: with $\varphi_q:=\tfrac1q\dn\cdot^q$,
\begin{equation}\tag{S2}\label{eq:S2}
\tfrac1q\dn{y}^q\;\ge\;\tfrac1q\dn{y'}^q+\ip{\nabla\varphi_q(y')}{y-y'}+\tfrac{\mu}{q}\dn{y-y'}^q
\qquad\text{for all }y,y' .
\end{equation}
(The conversion itself is in Remark~\ref{rem:hilbert}.) This appendix records
which $\ell_p$ and Schatten-$p$ norms satisfy \eqref{eq:S2}, and with which
$(q,\mu)$.

\begin{proposition}\label{prop:lp}
Let $\dn\cdot$ be the $\ell_p$ norm on $\R^m$ or the Schatten-$p$ norm on
$\R^{m\times n}$.
\begin{enumerate}[leftmargin=1.8em]
\item[(i)] For $1<p\le2$, \eqref{eq:S2} holds with $q=2$ and $\mu=p-1$: the
function $\psi_p:=\tfrac12\dn\cdot^2$ is $(p-1)$-strongly convex with respect
to $\dn\cdot$.
\item[(ii)] For $2\le p<\infty$, \eqref{eq:S2} holds with $q=p$ and
$\mu=2^{1-p}$.
\item[(iii)] For $p>2$, \eqref{eq:S2} fails with $q=2$ for every $\mu>0$: the
ratio $D_{\psi_p}(y,x)/\dn{y-x}^2$ can be made arbitrarily small.
\end{enumerate}
For $p=1$ (the $\ell_1$ and nuclear norms) the unit ball is not strictly convex
and no inequality of the form \eqref{eq:S2} can hold.
\end{proposition}

\begin{proof}
(i) Theorem~1 of \citet{ball1994sharp} (optimal $2$-uniform convexity) states
that for $1\le p\le2$ and all $X,Y$,
\begin{equation}\label{eq:bcl}
\Big(\frac{\dn{X+Y}^p+\dn{X-Y}^p}{2}\Big)^{2/p}\;\ge\;\dn X^2+(p-1)\dn Y^2 ,
\end{equation}
both for $\ell_p$ (their case (a), functions in $L_p$) and for Schatten-$p$
(their case (b), matrices in $C_p$). Put $X=\tfrac{a+b}{2}$ and
$Y=\tfrac{a-b}{2}$, so that $X+Y=a$ and $X-Y=b$. Because $2/p\ge1$, the map
$s\mapsto s^{2/p}$ is convex on $[0,\infty)$ and Jensen's inequality gives
$\big(\tfrac{\dn a^p+\dn b^p}{2}\big)^{2/p}\le\tfrac12\dn a^2+\tfrac12\dn b^2$.
Combining with \eqref{eq:bcl},
\[
\Big\|\frac{a+b}{2}\Big\|_*^2+(p-1)\Big\|\frac{a-b}{2}\Big\|_*^2\;\le\;\tfrac12\dn a^2+\tfrac12\dn b^2
\qquad\text{for all }a,b .
\]
Multiplying by $\tfrac12$ gives
$\psi_p(\tfrac{a+b}{2})\le\tfrac12\psi_p(a)+\tfrac12\psi_p(b)-\tfrac{p-1}{8}\dn{a-b}^2$,
which is midpoint strong convexity with modulus $p-1$; for continuous functions
this upgrades to strong convexity by the usual dyadic-subdivision argument, and
strong convexity of $\psi_p$ is \eqref{eq:S2} with $q=2$, $\mu=p-1$.

(ii) Let $\phi:=\dn\cdot^p$. Clarkson's inequality for
$p\ge2$ states that for all $y,y'$,
\begin{equation}\label{eq:clarkson}
\phi\Big(\frac{y+y'}{2}\Big)+\phi\Big(\frac{y-y'}{2}\Big)\le\frac{\phi(y)+\phi(y')}{2}.
\end{equation}
For $\ell_p$ this follows coordinatewise from the scalar case, which is
$X^{p/2}+Y^{p/2}\le(X+Y)^{p/2}\le\tfrac12(|a|^p+|b|^p)$ for
$X=(\tfrac{a+b}{2})^2$, $Y=(\tfrac{a-b}{2})^2$, $X+Y=\tfrac12(a^2+b^2)$, using
superadditivity of $s\mapsto s^{p/2}$ on $[0,\infty)$ (as $p/2\ge1$) and then
convexity of $s\mapsto s^{p/2}$; for Schatten norms \eqref{eq:clarkson} is
McCarthy's extension of Clarkson's inequality to trace classes
\citep{mccarthy1967cp}. Now let $m:=(y+y')/2$. Convexity of $\phi$ gives
$\ip{\nabla\phi(y')}{m-y'}\le\phi(m)-\phi(y')$, and $y-y'=2(m-y')$, so
\begin{align*}
D_\phi(y,y')&:=\phi(y)-\phi(y')-\ip{\nabla\phi(y')}{y-y'}
\ge\phi(y)-\phi(y')-2\big(\phi(m)-\phi(y')\big)\\
&=\phi(y)+\phi(y')-2\phi(m)\ge2\phi\Big(\frac{y-y'}{2}\Big)=2^{1-p}\dn{y-y'}^p ,
\end{align*}
the last inequality being \eqref{eq:clarkson}. Dividing by $p$ gives
\eqref{eq:S2} for $\varphi_p=\phi/p$ with $q=p$ and $\mu=2^{1-p}$.

(iii) Take $x=e_1$ and $y=e_1+te_2$ with $t\downarrow0$ (for Schatten norms,
$x=E_{11}$, $y=E_{11}+tE_{22}$). Since
$\nabla\psi_p(g)=\|g\|_p^{2-p}\big(\sign(g_i)|g_i|^{p-1}\big)_i$ we have
$\nabla\psi_p(e_1)=e_1$, hence $\ip{\nabla\psi_p(x)}{y-x}=0$ and
$D_{\psi_p}(y,x)=\tfrac12(1+t^p)^{2/p}-\tfrac12=\tfrac1pt^p+O(t^{2p})$ while
$\|y-x\|_p^2=t^2$, so the ratio is $\tfrac1pt^{p-2}+O(t^{2p-2})\to0$ for $p>2$.
Finally, if the unit ball is not strictly convex there are $y\ne y'$ on the
unit sphere with $\dn{(y+y')/2}=1$, and \eqref{eq:S2} applied at $y$ and $y'$
around the midpoint forces $\mu=0$.
\end{proof}

The constant $\mu=2^{1-p}$ in (ii) is not sharp (numerically the best constant
for $\ell_p$ is about $0.59$, $0.33$ and $0.10$ at $p=3,4,6$, against $0.25$,
$0.125$ and $0.031$), but it is explicit. Statement (iii) is why the power
type $q=p$, rather than $2$, appears in Corollary~\ref{cor:iterwise-ca} for
$p>2$: the unit sphere of $\ell_p$ is flatter than quadratic at the coordinate
points.

\section{Inexact oracles}\label{app:inexact}

The analysis uses the oracle only through $\pn{D_t}\le R$ and
$\ip{G_t}{D_t}=-R\dn{G_t}$. In the spectral case the polar factor is computed
by a Newton--Schulz iteration and is only approximate. Suppose instead that for
some $\vartheta\in[0,1)$ and all $t$,
\begin{equation}\label{eq:inexact}
\pn{D_t}\le(1+\vartheta)R,\qquad \ip{G_t}{D_t}\le-(1-\vartheta)R\,\dn{G_t},
\end{equation}
and that $\max_i\dn{\nabla f_i(x_t)}\le M$ along the trajectory (we do not
re-derive $M$ here). Then $\|g_t\|_\infty\le(1+\vartheta)RM$, and the two
places where the oracle identity enters the deterministic analysis change as
follows.

\begin{proposition}\label{prop:inexact}
Under Assumptions~\ref{as:smooth}--\ref{as:lb} and \eqref{eq:inexact}, with
$\tau=\kappa=0$, $\gamma R=\sqrt{2\Delta/(LT)}$ and $\eta R=\sqrt{2\ln K}/(M\sqrt T)$:
\begin{align*}
\frac1T\sum_{t<T}\dn{\gF(x_t)w_t}&\le\frac{(1+\vartheta)^2}{1-\vartheta}\cdot\frac{\sqrt{2\Delta L}+M\sqrt{\ln K/2}}{\sqrt T},\\
\frac1T\sum_{t<T}\big[\dn{\gF(x_t)w_t}-\dn{\gF(x_t)u}\big]&\le\frac{(1+\vartheta)^2}{1-\vartheta}M\sqrt{\frac{2\ln K}{T}}+\frac{2\vartheta M}{1-\vartheta}
\qquad\text{for every }u\in\simplex .
\end{align*}
\end{proposition}

\begin{proof}
\emph{One-step inequality.} The proof of Lemma~\ref{lem:onestep} gives
$F(x_{t+1})w\le F(x_t)w+\gamma\ip{A_tw}{D_t}+\tfrac12L\gamma^2\pn{D_t}^2$; with
\eqref{eq:inexact}, $\ip{A_tw}{D_t}=\ip{G_t}{D_t}+\ip{g_t}{w-w_t}\le-(1-\vartheta)R\dn{G_t}+\ip{g_t}{w-w_t}$
and $\pn{D_t}^2\le(1+\vartheta)^2R^2$. Repeating the proof of
Theorem~\ref{thm:det} with $u=u_K$ and $\|g_t\|_\infty\le(1+\vartheta)RM$,
\[
(1-\vartheta)\frac1T\sum_t\dn{G_t}\le\frac{\Delta}{\gamma RT}+(1+\vartheta)^2\frac{\eta RM^2}{2}+(1+\vartheta)^2\frac{L\gamma R}{2}
\le(1+\vartheta)^2\Big[\frac{\Delta}{\gamma RT}+\frac{\eta RM^2}{2}+\frac{L\gamma R}{2}\Big],
\]
and the step sizes give the first claim.

\emph{Regret.} Write $\psi_t(w)=\dn{A_tw}$. For any $u\in\simplex$, H\"older
gives $\dn{A_tu}\pn{D_t}\ge-\ip{A_tu}{D_t}=-\ip{g_t}{u}$, hence
$\psi_t(u)\ge-\ip{g_t}{u}/((1+\vartheta)R)$ (trivially if the right-hand side is
negative). For $w_t$, \eqref{eq:inexact} gives
$\psi_t(w_t)=\dn{G_t}\le-\ip{G_t}{D_t}/((1-\vartheta)R)=-\ip{g_t}{w_t}/((1-\vartheta)R)$.
Subtracting, with $a:=1/((1-\vartheta)R)$ and $b:=1/((1+\vartheta)R)$,
\begin{align*}
\psi_t(w_t)-\psi_t(u)&\le-a\ip{g_t}{w_t}+b\ip{g_t}{u}
=a\ip{g_t}{u-w_t}-(a-b)\ip{g_t}{u}\\
&\le\frac{\ip{g_t}{u-w_t}}{(1-\vartheta)R}+(a-b)(1+\vartheta)RM ,
\end{align*}
using $|\ip{g_t}{u}|\le\|g_t\|_\infty\le(1+\vartheta)RM$. Since
$(a-b)(1+\vartheta)RM=M(1+\vartheta)\big[\tfrac1{1-\vartheta}-\tfrac1{1+\vartheta}\big]=\tfrac{2\vartheta M}{1-\vartheta}$,
summing and applying Lemma~\ref{lem:md} as in the proof of
Theorem~\ref{thm:regret}, with $\|g_t\|_\infty^2\le(1+\vartheta)^2R^2M^2$, gives
the second claim.
\end{proof}

The multiplicative factor $(1+\vartheta)^2/(1-\vartheta)$ is controlled
when $\vartheta$ is bounded away from one. The second inequality has an
additive floor because of the mismatch between the upper and lower oracle
constants. This is a limitation of that comparator argument, not an
unavoidable floor for the average gap or regret along the descent trajectory.
Indeed, if $B_T$ denotes the first bound in Proposition~\ref{prop:inexact},
\eqref{eq:gap-stationarity} gives
\[
\begin{aligned}
\frac1T\sum_t\mathrm{gap}_t&\le B_T,\\
\frac1T\left[\sum_t\dn{A_tw_t}-\min_{u\in\simplex}\sum_t\dn{A_tu}\right]
&\le B_T=O(T^{-1/2}).
\end{aligned}
\]
Thus fixed relative oracle error satisfying \eqref{eq:inexact} is already
tolerated for deterministic stationarity and average inner optimality.

We do not extend this proposition to the stochastic or iteration-wise
results. There, oracle error must be propagated through the tracking and
subgradient arguments. Moreover, the error parameters of a practical
Newton--Schulz implementation depend on its polynomial, normalization,
coefficient schedule, numerical precision, and input spectrum. We have not
established a uniform $\vartheta<1$ for the five-step quintic implementation
used in the experiments, nor a geometric error-decay claim for that particular
implementation. Relating it to the oracle assumptions and extending the
stochastic and iteration-wise guarantees remain open.

\section{Experimental details}\label{app:expdetails}

\subsection{Cityscapes and NYU-v2}\label{app:densedetails}

\paragraph{Protocol.} We follow the protocol of
\citet{navon2022nashmtl,liu2023famo}: an MTAN model on a SegNet backbone, with 200 epochs, and batch sizes 8 on Cityscapes and 2 on NYU-v2.
The standard baselines use Adam with learning rate $10^{-4}$, which is multiplied by $0.5$ after epoch 100. Our method applies the oracle direction directly with step size $\gamma$, while MOON uses the released Muon+AdamW optimizer. Reported
metrics are averaged over the final 10 epochs with one seed; these epochs
are correlated and do not estimate variability across independent runs. The baseline rows
are re-runs with the same protocol. Cityscapes has $K=2$ tasks
(7-class semantic segmentation, depth); NYU-v2 has $K=3$ (13-class
segmentation, depth, surface normals).

\paragraph{Metrics.} Segmentation: mean IoU and pixel accuracy (higher is
better). Depth: absolute and relative error (lower is better). Normals: mean
and median angular error (lower is better) and the fraction of pixels within
$11.25^\circ$, $22.5^\circ$ and $30^\circ$ (higher is better). The aggregate is
\[
\Delta m\%=\frac{100}{n}\sum_{j=1}^{n}(-1)^{s_j}\,\frac{M_j-M_{\mathrm{STL},j}}{M_{\mathrm{STL},j}},
\]
over the $n$ metrics, with $s_j=1$ if higher is better and $0$ otherwise.
For each dataset, $M_{\mathrm{STL}}$ is the corresponding single-task
baseline reported in the STL row of the same table. Thus, the STL reference
is $(74.11,93.44,0.0137,38.22)$ on Cityscapes and
$(38.77,63.98,0.6350,0.2654,25.10,19.06,30.12,57.66,69.53)$ on NYU-v2. Accordingly, the STL row has $\Delta m=0$, and all other $\Delta m$ values
are computed relative to these references.

\paragraph{Our method.}
We use Algorithm~\ref{alg:stoch} with the spectral oracle on convolutional and
linear weights (4D kernels flattened over their last three dimensions) and the
Euclidean oracle on the remaining parameters. Each matrix-valued parameter
tensor is treated as a unit-radius spectral-norm block, $\rho_\ell=1$, without
width- or kernel-dependent rescaling, while all scalar and vector parameters
are concatenated into a single unit-radius Euclidean block. The outer model
step size controls the overall update magnitude. We use seed $42$ for all
dense-task experiments, with momentum $\alpha=0.1$, $\tau=0$, and $\kappa=0$
fixed throughout.

Hyperparameter tuning uses the same non-training benchmark split on which the
dense-task results are reported: the official 500-image validation split for
Cityscapes and the 654-image evaluation (test) split for NYU-v2. Thus, no
additional held-out validation split is used for these dense-task experiments.
Short-horizon segmentation performance is used as the primary tuning criterion,
subject to maintaining joint-task performance, while the final late-training
schedule is selected using the final-window $\Delta m$.
For Cityscapes, we search
$\gamma\in\{10^{-3},5\times10^{-3}\}$ and
$\eta\in\{10^{-4},5\times10^{-4}\}$, selecting
$\gamma=5\times10^{-3}$ and $\eta=10^{-4}$. The final model-step schedule uses
$5\times10^{-3}$ for epochs 1--100, $2.5\times10^{-3}$ for epochs 101--150,
and $1.25\times10^{-3}$ for epochs 151--200. For NYU-v2, we search
$\gamma\in\{2.5\times10^{-3},5\times10^{-3}\}$ and
$\eta\in\{5\times10^{-5},10^{-4}\}$, selecting
$\gamma=2.5\times10^{-3}$ and $\eta=10^{-4}$; the model step is reduced to
$1.25\times10^{-3}$ from epoch 101 onward. Standard baselines use their
reference hyperparameters without an additional tuning grid.

\subsection{MultiMNIST}\label{app:mnistdetails}

\paragraph{Protocol.} Following \citep{moon2026}, we use MultiMNIST, where each
example contains two MNIST digits and the two tasks are to classify the left
and right digits. The digits are concatenated horizontally and, with
probability $0.5$, overlapped by up to $30\%$ of their width. The resulting
images are resized to $28\times43$. We generate 10{,}000 training examples from
MNIST-train and 1{,}000 test examples from MNIST-test. We use a small Vision
Transformer with patch size $4$ and embedding dimension $16$, consisting of a
class token, learned positional embeddings, and $3$ transformer encoder layers,
each with $2$ attention heads and a feed-forward dimension of $64$. No dropout
is used, and the shared encoder is followed by one linear classification head
for each task. All methods in Table~\ref{tab:mnist} are evaluated on the same
generated dataset and trained for 100 epochs with batch size $256$.

\paragraph{Our method.} Algorithm~\ref{alg:stoch} with the unregularized,
unclipped specialization of the weight update in Algorithm~\ref{alg:main}. At each iteration, we compute
per-task gradients for all trainable model parameters and form their weighted
aggregate. The oracle comes from the previous momentum buffer.
For every parameter tensor with at least two dimensions, we flatten all but
the first dimension and apply five Newton--Schulz steps. If the resulting
matrix is taller than wide, we first transpose it. Writing $G$ for this
(possibly transposed) matrix, the iteration is carried out in \texttt{float32},
initialized as $X_0=G/(\|G\|_F+10^{-7})$, and uses
\[
X_{r+1}=aX_r+b(X_rX_r^\top)X_r+c(X_rX_r^\top)^2X_r,
\qquad (a,b,c)=(3.4445,-4.7750,2.0315),
\]
for five steps. Each
matrix-valued parameter tensor is treated as a unit-radius spectral-norm block,
without width- or kernel-dependent rescaling, while all scalar and vector
parameters are concatenated into a single unit-radius Euclidean block. The
blended momentum buffer is initialized at zero; its first update uses
$\alpha_0=1$, and subsequent updates use $\alpha=0.1$. We use model step size
$\gamma=10^{-3}$, weight step $\eta=5\times10^{-5}$, $\tau=0$, and $\kappa=0$.

\paragraph{MOON.} We use \citep{moon2026}'s released implementation: the task logits are updated with Adam at learning rate
$10^{-4}$ and weight decay $10^{-3}$, Muon uses Nesterov momentum, and tasks
are scored using the previous optimizer update. The data, model, number of
epochs, batch size and numerical model learning rate are shared with our
method. Momentum conventions, logit optimizers, weight decay, and auxiliary
parameter updates differ; a common numerical learning rate need not imply a
common effective update size. The comparison therefore evaluates complete
implementations rather than isolating a single algorithmic component.

\paragraph{Other baselines.} MGDA and FAMO use Adam, as in MOON's experiments;
``+ Muon'' replaces Adam on the encoder matrices by Muon while keeping the
baseline's task weights, and Muon (LS) uses Muon with equal weights.
In all MultiMNIST runs the model learning rate is $10^{-3}$; MGDA uses its
minimum-norm task weights, FAMO uses weight learning rate $0.025$ and
$\gamma=0.01$, and the Muon variants use momentum $0.95$, five Newton--Schulz
steps, RMS matching $0.2$, and weight decay $10^{-3}$, with auxiliary Adam
parameters using $(\beta_1,\beta_2)=(0.9,0.95)$.

\paragraph{Tuning and reporting.} All hyperparameter selection is performed on a fixed $10\%$ validation split drawn from MNIST-train using split seed $2026$. The weight-step sweep evaluates $\eta\in\{10^{-5},5\times10^{-5},10^{-4},5\times10^{-4},10^{-3}\}$ (Table~\ref{tab:eta}), together with the other candidate configurations and schedules considered during tuning. Final candidates are evaluated over seeds $42$, $43$, and $44$ using the mean validation accuracy over epochs 96--100. After all hyperparameters are frozen, the selected configuration is evaluated on MNIST-test.

\paragraph{Ablations.} Each ablation changes one component of our method and
keeps the others, including the step sizes tuned for the spectral
configuration, on three (42, 43, and 44) seeds. These are fixed-hyperparameter
substitutions, not comparisons of separately tuned variants. In particular,
the oracle results do not establish superiority over a retuned Euclidean
oracle, and the weight results do not establish equivalence of entropic and
projected updates. For the oracle ablations, the sign variant uses the
unit-radius $\ell_\infty$ oracle $\mathrm{sign}(M)$, while the $\ell_2$
variant uses the unit-radius Euclidean oracle $M/\|M\|_2$ on the full parameter
vector. The spectral configuration follows the product geometry of the main
method: each matrix-valued parameter tensor forms a unit-radius spectral-norm
block, while all scalar and vector parameters are concatenated into a single
unit-radius Euclidean block. No additional RMS matching or $0.2$ rescaling is
applied to these oracle variants.

\emph{Weights}: ``projected'' replaces the entropic step by the
Euclidean projection of $w_t+\eta g_t$ onto $\simplex^\kappa$. \emph{Momentum}:
``per-task'' keeps one buffer $m_{t,i}$ per task, evaluates the oracle at
$\sum_iw_{t,i}m_{t,i}$ and scores with $\langle D_t,m_{t,i}\rangle$
(Appendix~\ref{app:pertask}); ``none'' evaluates the oracle at the current
blended minibatch gradient. \emph{Step $\eta$}: the grid above, all other
settings fixed.

\begin{table}[h]\centering\small
\begin{tabular}{lccccc}
\toprule
$\eta$ & $10^{-5}$ & $5\times10^{-5}$ & $10^{-4}$ & $5\times10^{-4}$ & $10^{-3}$\\
\midrule
Average accuracy (\%) & 90.20 & \textbf{90.54} & 90.08 & 89.69 & 89.53\\
Final inner gap & 1.97 & 0.90 & 0.84 & 0.29 & 1.04\\
\bottomrule
\end{tabular}
\caption{MultiMNIST: sensitivity of our method to the weight step $\eta$ (spectral
oracle, entropic weights, blended momentum), test accuracy averaged over three
seeds. The main configuration, $\eta=5\times10^{-5}$, was selected on the
validation split.}
\label{tab:eta}
\end{table}

\begin{table}[!ht]
\centering
\scriptsize
\setlength{\tabcolsep}{5pt}
\begin{tabular}{lccccc}
\toprule
$\kappa$ & 0 & 0.48 & 0.49 & 0.495 & 0.4975 \\
\midrule
Avg.$\uparrow$ 
& 90.54 
& 90.54 
& 90.73 
& \textbf{91.11} 
& 90.96 \\
Gap$\downarrow$ 
& \textbf{0.90} 
& \textbf{0.90} 
& 1.17 
& 1.22 
& 1.65 \\
\bottomrule
\end{tabular}
\caption{Ablation of the clipping parameter $\kappa$ on MultiMNIST: test accuracy and final inner gap, averaged over three seeds.}
\label{tab:kappa_ablation}
\end{table}

\paragraph{Clipping.} Table~\ref{tab:kappa_ablation} varies the floor $\kappa$
with all other settings of the main configuration fixed ($\eta=5\times10^{-5}$,
$\tau=0$). For $\kappa\le0.48$ the results coincide with the unclipped run to
the reported precision, so the floor essentially never binds: with this weight
step the weights remain within $[0.48,0.52]$. As $\kappa$ approaches
$1/K=0.5$ the floor forces the weights toward uniform; the mean accuracy rises
slightly, to $91.11\%$ at $\kappa=0.495$, while the empirical inner gap grows
from $0.90$ to $1.65$. Since $\kappa\to1/K$ reduces the method to linear
scalarization with the same oracle and momentum, and this configuration
reaches a similar accuracy, the margin over the baselines in
Table~\ref{tab:mnist} cannot be attributed to the adaptive weighting at $K=2$.

\paragraph{Empirical inner-gap diagnostic.}
At the end of training, we compute per-task gradients of all trainable
parameters on a fixed probe consisting of the first $256$ training examples
and report the simplex Frank--Wolfe gap of
\[
f(w)=\left\|\sum_i w_i\nabla \ell_i\right\|_* 
\]
at the final task weights. Specifically, writing
$G(w)=\sum_i w_i\nabla\ell_i$ and letting $P(G(w))$ denote the exact
geometry-specific polar direction, we report
\[
\mathrm{gap}
=
\left\langle G(w),P(G(w))\right\rangle
-
\min_i
\left\langle \nabla\ell_i,P(G(w))\right\rangle .
\]
For the spectral configuration, the exact diagnostic uses the product geometry
of the method: each matrix-valued parameter block has unit spectral-norm
radius, whose dual contribution is the nuclear norm, while all scalar and
vector parameters form a single unit-$\ell_2$ block. The sign oracle corresponds
to a unit $\ell_\infty$ ball over the full parameter vector and hence an
entrywise $\ell_1$ dual norm. The $\ell_2$ ablation uses one unit Euclidean
ball over the full parameter vector, so its dual is the global $\ell_2$ norm,
rather than a sum of per-layer Frobenius norms. Because the underlying
geometries differ, the reported gaps for different oracle variants are not on
a common numerical scale. This is a fixed-minibatch diagnostic rather than the
exact full-dataset inner optimization error appearing in the theoretical
analysis. All reported main configurations use $\tau=\kappa=0$ and therefore
do not directly test the regularized iteration-wise tracking guarantee.

\paragraph{Code.} An anonymous implementation of all methods and experiments is available at
\url{https://anonymous.4open.science/r/LMO-MOO-D648/}.

\end{document}